\documentclass[runningheads,envcountsame,numbook]{svjour3}

\usepackage{amsmath,amssymb,mathptmx}
\usepackage{cite}
\usepackage[width=12cm, left=4.5cm, height=23cm] {geometry}
\usepackage{color}
\usepackage{hyperref}
\hypersetup{
    colorlinks=true,
    linkcolor=blue, 
    citecolor=red, 
    urlcolor=blue  } 
\usepackage[weather,alpine,misc,geometry]{ifsym}
\smartqed  

\newcommand{\al}{\alpha}
\newcommand{\be}{\beta}

\newcommand{\la}{\lambda}
\newcommand{\de}{\delta}
\newcommand{\eps}{\varepsilon}
\newcommand{\bx}{\bar x}
\newcommand{\by}{\bar y}

\newcommand{\iv}{^{-1} }

\newcommand {\R} {\mathbb R}

\newcommand {\B} {\mathbb B}

\newcommand {\gph} {{\rm gph}\,}
\newcommand {\dom} {{\rm dom}\,}

\newcommand {\Er} {{\rm Er}\,}

\newcommand {\sd} {\partial}

\newcommand{\folgt}{$ \Rightarrow\ $}

\def\nbh{neighbourhood}
\def\es{\emptyset}
\def\lsc{lower semicontinuous}

\def\LHS{left-hand side}

\def\SVM{set-valued mapping}
\def\EVP{Ekeland variational principle}
\def\Fr{Fr\'echet}

\newcommand{\norm}[1]{\left\Vert#1\right\Vert}

\newcommand{\ang}[1]{\left\langle #1 \right\rangle}
\newcommand{\qdtx}[1]{\quad\mbox{#1}\quad}

\newcounter{mycount}


\newcommand {\Erq} {{\rm Er_q}\,}
\newcommand {\clmq} {{\rm clm_q}\,}

\title
{H\"{o}lder Error Bounds and H\"{o}lder Calmness with Applications to Convex Semi-Infinite Optimization
\thanks{The research of the first and second authors was supported by the Australian Research Council, project DP160100854.
{The first author benefited from the support of the FMJH Program PGMO and from the support of EDF.}
The research of the second author was also supported by MINECO of Spain
and FEDER of EU, grant MTM2014-59179-C2-1-P.
The research of the third author was supported by the Research Grants Council of Hong Kong (PolyU 152342/16E).
The research of the fourth author was supported by the National Natural Science Foundation of China, grants 11771384 and 11461080.
}
}

\titlerunning
{H\"{o}lder Error Bounds and H\"{o}lder Calmness}

\author{
Alexander Y. Kruger
\and Marco A. L\'{o}pez
\and Xiaoqi Yang
\and Jiangxing Zhu
}
\institute{
Alexander Y. Kruger (\Letter\,) \at
Centre for Informatics and Applied Optimization, Federation University Australia, Ballarat, Australia\\
\email{a.kruger@federation.edu.au}
\and
Marco A. L\'{o}pez \at
Department of Mathematics, University of Alicante, Alicante, Spain and Centre for Informatics and Applied Optimization, Federation University Australia, Ballarat, Australia\\
\email{marco.antonio@ua.es}
\and
Xiaoqi Yang \at
Department of Applied Mathematics, The Hong Kong Polytechnic University, Hong Kong\\
\email{xiao.qi.yang@polyu.edu.hk}
\and 
Jiangxing Zhu \at
Department of Mathematics, Yunnan University, Kunming 650091, China\\
\email{jiangxingzhu@yahoo.com}
}
\date{Received: date / Accepted: date}

\journalname{}

\begin{document}

\maketitle

\begin{abstract}
Using techniques of variational analysis,
necessary and sufficient subdifferential conditions for H\"{o}lder error bounds are investigated and some new estimates for the corresponding modulus are
obtained.
As an application, we
{consider the setting of convex semi-infinite optimization and}
give a characterization of the H\"{o}lder calmness of the argmin
mapping in terms of the level set mapping (with respect to the objective
function) and a special supremum function. We also estimate the H\"{o}lder
calmness modulus of the argmin mapping in the framework of linear
programming.
\end{abstract}




\section{Introduction}

This paper mainly concerns the study and some applications of the notions of
H\"{o}lder error bounds and H\"{o}lder calmness.

Given
an extended-real-valued function $f:X\to\mathbb{R}\cup\{+\infty\}$ on a metric space $X$, a point $\bar x\in[f\leq0]:=\{x\in X\mid f(x)\leq0\}$ and a number $q>0$,
we say that $f$ admits a \emph{$q$-order local error bound} at $\bar x $, if there exist positive numbers $\tau$ and $\delta$ such that
\begin{equation}\label{EB}
\tau d(x, [f\leq 0])\leq [f(x)]_+^q, \quad \forall x\in B_\de(\bx),
\end{equation}
where
$[f(x)]_+:=\max\{f(x),0\}$.

The supremum of all $\tau>0$ in \eqref{EB} is called the \emph{modulus}
{(\emph{conditioning rate} \cite{Pen13})}
of $q$-order error bounds of $f$ at $\bx$ and is denoted by $\Erq f(\bx)$;
explicitly,
\begin{equation}\label{Er}
\Erq f(\bx)
:=\liminf_{x\to\bar x,\;f(x)>0} \frac{f(x)^q}{d(x,[f\leq 0])}.
\end{equation}
It provides a quantitative characterization of error bounds.
The absence of error bounds, i.e., the situation when \eqref{EB} does not hold for any $\tau>0$, is signaled by $\Erq f(\bx)=0$.
When $q=1$, we write simply $\Er f(\bx)$ instead of Er$_1\,f(\bx)$.

The case $q=1$ corresponds to the conventional \emph{linear error bounds}.
Linear error bounds have been well studied, especially in the last 15 years, because of numerous applications; see, e.g., \cite{Aze03,AzeCor04,FabHenKruOut10,FabHenKruOut12, KruNgaThe10,Iof79,Kru15,MenYan12,NgaThe09, Jou00,JouYe05,Ye98,WuYe02,WuYe03}.
The study of H\"{o}lder ($q$-order) and more general nonlinear error bounds started relatively recently thanks to more subtle applications, where conventional linear estimates do not hold; see \cite{ZhaNgZheHe16,YaoZhe16,NgaTroTin17,WuYe02, Iof17}.

Many authors have studied seemingly more general than error bounds, but in a sense equivalent to them properties of
nonlinear subregularity and
calmness of \SVM s, which
are
of great importance
for
optimization
as well as
subdifferential calculus, optimality conditions,
stability and sensitivity issues, convergence of numerical methods, etc;
interested readers are referred to \cite{Kru15.2,Kru16,NgaThe09,LiMor12,YaoZhe16, ZhaNgZheHe16,GayGeoJea11,MorOuy15,NgaTroTin17, KlaKruKum12,Kum09,Kla94, NgaTroThe16} and the references
therein.
Sufficient conditions for (nonlinear) error bounds generate sufficient conditions for (nonlinear) subregularity and calmness; see, e.g., \cite{Kru15,Kru15.2,Kru16}.

Given
a \SVM\ $S:Y\rightrightarrows X$ between metric spaces $Y$ and $X$, a point $(\by,\bx)\in\gph(S)
$ and a number $q>0$,
we say that $S$ is $q$-\emph{order calm}
(or possesses $q$-\emph{order calmness} property)
at $(\bar{y},\bar{x})$ if there exist a number $\tau>0$ and neighborhoods $U$ of $\bar{x}$ and $V$ of $\bar{y}$ such that
\begin{equation}
\tau d(x,S(\bar{y}))\leq d(y,\bar{y})^{q},\quad \forall y\in V\;\mbox{and}\; x\in
S(y)\cap U.  \label{1.1}
\end{equation}
If, additionally, $\bx$ is an isolated point in $S(\bar y)$, i.e., $S(\bar y)\cap U=\{\bar x\}$, then we say that $S$ possesses $q$-\emph{order isolated calmness} property

The supremum of all $\tau>0$ in \eqref{1.1} is called the $q$-\emph{order calmness modulus} of $S$ at $(\bar{y},\bar{x})$ and is denoted by $\clmq S(\bar{y},\bar{x})$;
explicitly,
\begin{equation}
\clmq S(\bar{y},\bar{x}):=\liminf_{\substack{ y\rightarrow \bar{y}  \\ x\rightarrow \bar{x},\;x\in S(y)}}\frac{d(y,\bar{y})^{q}}{%
d(x,S(\bar{y}))}
=\liminf_{x\rightarrow \bar{x}}\frac{%
d(\bar{y},S^{-1}(x))^q}{d(x,S(\bar{y}))}.  \label{1.6}
\end{equation}
It provides a quantitative characterization of the calmness property.
Following the lines
of
\cite[Theorem 2.2]{DonRoc04}, it is
easy
to verify that $\clmq S(\bar{y},\bar{x})$ coincides with the modulus of \emph{$q$-order metric subregularity} of $%
S^{-1}$ at $(\bar{x},\bar{y})$.

Using techniques of variational analysis, we continue the study of necessary and sufficient subdifferential conditions for H\"{o}lder error bounds, particularly merging the conventional approach with the new advancements proposed recently in \cite{YaoZhe16}.
We formulate a general lemma collecting the main arguments used in the proofs of the subdifferential sufficient error bound conditions and demonstrate that both linear and H\"older type conditions, both conventional and those in \cite{YaoZhe16}, can be obtained as direct consequences of this lemma.

Moreover, we
{clarify the relationship between the error bound characterizations in \cite{YaoZhe16} and}
those obtained using the conventional approach.
Some new estimates for the modulus of $q$-order error bounds are obtained.
The main sufficient subdifferential conditions are combinations of two assertions: in Asplund spaces in terms of Fr\'echet subdifferentials and for convex functions in general Banach spaces.


In \cite{CanKruLopParThe14}, the authors compute or estimate the calmness modulus of the
argmin mapping of linear semi-infinite optimization problems under canonical
perturbations (see Section~\ref{S5} for its explicit meaning). Motivated by this
and as
an application,
the last goal of the paper is to study in detail H\"{o}%
lder calmness in convex semi-infinite programs. In this setting we clarify
the relationship among the H\"{o}lder calmness of the solution mapping $%
\mathcal{S}$, the (lower) level set mapping $\mathcal{L}$, and the H\"{o}%
lder error bound of a special supremum function (see their definitions in
Section~\ref{S5}). Moreover, we also estimate the H\"{o}lder calmness modulus of
the argmin mapping in the framework of linear programming.

The rest of the paper is organized as follows: Section~\ref{S2} summarizes some
preliminary facts from variational analysis and generalized differentiation
widely used in the formulations and proofs of our main results.
In Section~\ref{S3} we establish and discuss some necessary and sufficient subdifferential conditions for H\"{o}lder error bounds.
The last
Section~\ref{S5}
devoted to convex semi-infinite optimization, shows the equivalence among the H\"{o}lder calmness of the (lower) level set and argmin mappings and the H\"{o}lder error bounds of a special supremum function, and also provides an
estimate
of
H\"{o}lder calmness of the argmin mapping under some particular
conditions.

{The paper is dedicated to
our friend Professor Asen Dontchev on the occasion of his 70$^{th}$ birthday}

\section{Preliminaries}\label{S2}

In this section, we
summarize some fundamental tools of variational analysis and nonsmooth
optimization.

Our basic terminology and notation are standard, see, e.g., \cite{Mor06.1, Cla83,ClaLedSteWol98,RocWet98,Sch07,DonRoc14,Iof17}.
Throughout the paper, $X$ and $Y$ are either metric or normed vector spaces.
We use the standard notations $d(\cdot,\cdot)$ and $\norm{\cdot}$ for the distance and the norm in any space.
For $x\in X$ and $\delta>0$, $B_\de(x)$
denotes the open ball
centered at $x$ with radius $\delta$.
Given a set $A$ and a point $x$ in the same space,
$d(x,A):=\inf_{a\in A}d(x,a)$ is
the distance from $x$ to $A$.
In particular, $d(x,\emptyset)=+\infty$ for any $x$.
If $X$ is a normed vector space, its
topological dual is denoted by $X^*$, while $\langle\cdot,\cdot\rangle$ denotes the bilinear form defining the pairing between the two spaces.
We denote by $\B$ and $\B^*$ the open unit
balls in a normed vector space and its dual, respectively.

Given an extended-real-valued function $f: X \rightarrow \mathbb{R}\cup
\{+\infty\}$, we denote by $\dom f$
its domain:
$\dom f:=\{x\in X\mid f(x)<+\infty\}$.
For a
set-valued mapping $\varPhi:X\rightrightarrows Y$,
the graph and the domain of $\varPhi$ are defined, respectively, by
\begin{equation*}
\gph(\varPhi):=\{(x,y)\in X\times Y\mid x\in X,\; y\in \varPhi(x)\} \quad %
\mbox{and}\quad \dom\varPhi:=\{x\in X\mid\varPhi(x)\neq \emptyset\}.
\end{equation*}
The inverse $F\iv:Y\rightrightarrows X$ of $F$ (which always exists with possibly empty values at some $y$) is defined by
\begin{equation*}
F\iv(y):=\{x\in X\mid y\in F(x)\},\quad y\in Y.
\end{equation*}
Obviously, $\dom F\iv=F(X)$.

Recall that the Fr\'{e}chet subdifferential of $f$ at $x\in\dom f$
is defined as
\begin{equation*}
{\partial}f(x):=\left\{x^*\in X^*\mid \;\liminf\limits_{y\rightarrow x}
\frac{f(y)-f(x)-\langle x^*,y-x\rangle}{\|y-x\|}\geq0\right\},
\end{equation*}
and
$
\partial f(x):=\emptyset $ if $x\notin\dom f$. It is
well-known that the set $
\partial f(x)$ reduces to the classical
subdifferential
of convex analysis
if $f$ is convex.

The proofs of the main results rely on certain fundamental results of variational analysis: the \emph{Ekeland variational principle} (Ekeland \cite{Eke74}; see also \cite{Mor06.1,DonRoc14,Iof17}) and several kinds of \emph{subdifferential sum rules}.
Below we provide these results for completeness.

\begin{lemma}[Ekeland variational principle] \label{Ek}
Suppose $X$ is a complete metric space, $f:X\to\R\cup\{+\infty\}$ is lower semicontinuous and bounded from below, $\varepsilon>0$ and $\lambda>0$. If
$$
f(\bx)<\inf_X f + \varepsilon,
$$
then there exists an $\hat x\in X$ such that

{\rm (a)} $d(\hat x,\bx)<\lambda $,

{\rm (b)} $f(\hat x)\le f(\bx)$,

{\rm (c)} $f(x)+(\varepsilon/\lambda)d(x,\hat x)\ge f(\hat x)$ for all $x\in X$.
\end{lemma}

\begin{lemma}[Subdifferential sum rules] \label{SR}
Suppose $X$ is a normed linear space, $f_1,f_2:X\to\R\cup\{+\infty\}$, and $\bx\in\dom f_1\cap\dom f_2$.

{\rm (i) \bf Fuzzy sum rule}. Suppose $X$ is Asplund,
$f_1$ is Lipschitz continuous and
$f_2$
is lower semicontinuous in a neighbourhood of $\bar x$.
Then, for any $\varepsilon>0$, there exist $x_1,x_2\in X$ with $\|x_i-\bar x\|<\varepsilon$, $|f_i(x_i)-f_i(\bar x)|<\varepsilon$ $(i=1,2)$, such that
$$
\partial (f_1+f_2) (\bar x) \subset \partial f_1(x_1) +\partial f_2(x_2) + \varepsilon\B^\ast.
$$


{\rm (ii) \bf Convex sum rule}. Suppose
$f_1$ and $f_2$ are convex and $f_1$ is continuous at a point in $\dom f_2$.
Then
$$
\partial (f_1+f_2) (\bar x) = \sd f_1(\bx) +\partial f_2(\bx).
$$
\end{lemma}

The first sum rule in the lemma above is known as the \emph{fuzzy} or \emph{approximate} sum rule (Fabian \cite{Fab89}; cf., e.g., \cite[Rule~2.2]{Kru03}, \cite[Theorem~2.33]{Mor06.1}) for Fr\'echet subdifferentials in Asplund spaces.
The other one is an example of an \emph{exact} sum rule.
It is valid in arbitrary normed spaces.
For rule (ii) we refer the readers to \cite[Theorem~0.3.3]{IofTik79} and
\cite[Theorem~2.8.7]{Zal02}.

Recall that a Banach space is \emph{Asplund} if every continuous convex function on an open convex set is Fr\'echet differentiable on a dense subset \cite{Phe93}, or equivalently, if the dual of each its separable subspace is separable.
We refer the reader to \cite{Phe93,Mor06.1,BorZhu05} for discussions about and characterizations of Asplund spaces.
All reflexive, in particular, all finite dimensional Banach spaces are Asplund.

The following fact is an immediate consequence of the definition of the \Fr\ subdifferential (cf., e.g., \cite[Propositions~1.10]{Kru03}).

\begin{lemma}\label{L2.3}
Suppose $X$ is a normed vector space and $f:X\to\R\cup\{+\infty\}$.
If $\bx\in\dom f$ is a point of local minimum of $f$, then $0\in\sd f(\bx)$.
\end{lemma}

The next subdifferential \emph{chain rule} is
a modification of \cite[Lemma~1]{NgaThe09} and \cite[Corollary~1.14.1]{Kru03}; see also \cite[Proposition~2.1]{YaoZhe16}.

\begin{lemma}\label{L2}
Suppose $X$ is a
normed linear
space, $f:X\rightarrow\mathbb{R}\cup\{+\infty\}$ is lower semicontinuous and $\bar x\in\dom f$.
Suppose also that $\psi:\mathbb{R}\to
\mathbb{R}\cup\{+\infty\}$ is
differentiable
at
$f(\bar x)$ with $\psi^{\prime }(f(\bar x))>0$ and
nondecreasing on $[f(\bar x),+\infty)$.
Then
\begin{equation*}
\partial(\psi\circ f)(\bar x)=\psi^{\prime }(f(\bar x))
\partial f(\bar x).
\end{equation*}
\end{lemma}

\section{Characterizations of H\"{o}lder error bounds} \label{S3}

In this section, we establish and discuss some necessary and sufficient subdifferential conditions for H\"{o}lder error bounds.
We start with a slightly new look at the very well studied linear error bounds.



\renewcommand {\theenumi} {\rm(\roman{enumi})}
\renewcommand {\labelenumi} {\theenumi}
\renewcommand {\theenumii} {\rm(\alph{enumii})}
\renewcommand {\labelenumii} {\theenumii}

\subsection{Linear error bounds}

The next
{elementary}
lemma collects the main arguments used in the proofs of the subdifferential sufficient error bound conditions, the key tools being the Ekeland variational principle (Lemma~\ref{Ek}) and subdifferential sum rules (Lemma~\ref{SR}).
It
{establishes an error bound estimate for a fixed point $x\notin[f\le0]$ and}
actually combines two separate statements: for \lsc\ functions in Asplund spaces and for convex functions in general Banach spaces.
All sufficient error bound conditions in this section are in a sense consequences of this lemma.

\begin{lemma}\label{L1}
Suppose $X$ is a Banach space, $f:X\rightarrow\mathbb{R}\cup\{+\infty\}$ is lower semicontinuous, $x\in X$ and $f(x)>0$.
Let $\tau>0$ and $\al\in(0,1]$.
If either $X$ is Asplund and,
{given an $M>f(x)$,}
\begin{multline}\label{L1-1}
d(0,\partial f(u))\ge\tau
\quad\mbox{for all}\quad
u\in X
\;\mbox{with}\;
\|u-x\|<\alpha d(x,[f\leq0]),\\
f(u)<M
\;\mbox{and}\;
f(u)<\tau d(u,[f\le0]),
\end{multline}
or $f$ is convex and
\begin{multline}\label{L1-1c}
d(0,\partial f(u))\ge\tau
\quad\mbox{for all}\quad
u\in X
\;\mbox{with}\;
\|u-x\|<\alpha d(x,[f\leq0]),\\
f(u)\le f(x)
\;\mbox{and}\;
f(u)<\tau d(u,[f\le0]),
\end{multline}
then $[f\leq0]\ne\es$ and
\begin{gather}\label{L1-2}
\alpha\tau d(x,[f\leq0])\leq f(x).
\end{gather}
\end{lemma}
\begin{proof}
Suppose that condition \eqref{L1-2} is not satisfied, i.e.
$f(x)<\alpha
\tau d(x,[f\leq0])$ (this is the case, in particular, when $[f\leq0]=\es$).
Choose a $\hat\tau\in(0,\tau)$ such that $f(x)<\alpha
\hat\tau d(x,[f\leq0])$.
Then, by
the Ekeland variational principle
{(Lemma~\ref{Ek})}
applied to the \lsc\ function $f_+$, there exists a point $\hat u\in X$ such that
\begin{gather}\label{L1P-1}
f_+(\hat u)\le f(x),\quad
\|\hat u-x\|<\alpha d(x,[f\leq0]),
\\\label{L1P-2}
f_+(\hat u)\le f_+(u)+
\hat\tau\|u-\hat u\|
\quad\mbox{for all}\quad
u\in X.
\end{gather}
Since $\alpha\in(0,1]$, it follows from \eqref{L1P-1} that $\hat u\notin[f\leq0]$, and consequently, $f_+(\hat u)=f(\hat u)>0$.
{If $[f\leq0]\ne\es$,}
it follows from \eqref{L1P-2} that
\begin{gather}\label{L1P-4}
f(\hat u)\le\hat\tau d(\hat u,[f\le0]).
\end{gather}
{If $[f\leq0]=\es$, the last inequality is satisfied trivially.}
In view of the lower semicontinuity of $f$, we have $f_+(u)=f(u)>0$ for all $u$ near $\hat u$, and it follows from \eqref{L1P-2} and Lemma~\ref{L2.3} that $0\in\partial(f+g)(\hat u)$ where $g(u):=\hat\tau\|u-\hat u\|$, $u\in X$.

\begin{enumerate}
\item
Suppose $X$ is Asplund.
Choose an $\varepsilon>0$ such that
\begin{multline}\label{L1P-3}
f(x)+\varepsilon<M,
\;\;
\|\hat u-x\|+\varepsilon<\alpha d(x,[f\leq0]),
\;\;
\varepsilon<f(\hat u),
\;\;
\varepsilon <\tau-\hat\tau
\\
\mbox{and}\;\;
\varepsilon(1+\tau)<(\tau-\hat\tau) d(\hat u,[f\le0]).
\end{multline}
Applying the fuzzy sum rule (Lemma~\ref{SR}(i)), we find
points $\hat x,\hat x'\in B_{\varepsilon}(\hat u)$ and subgradients $x^*\in\partial f(\hat x)$, $x'{}^*\in\partial g(\hat x')$
such that $|f(\hat x)-f(\hat u)|<\eps$
and $\|x^*+x'{}^*\|<\varepsilon$.
By the definition of $g$, we have $\|x'{}^*\|\le
\hat\tau$.
Using \eqref{L1P-3}, \eqref{L1P-4}, \eqref{L1P-1} and the obvious inequality $d(\hat x,[f\le0])+\|\hat x-\hat u\|-d(\hat u,[f\le0])\ge0$, we obtain the following estimates:
\begin{align*}\notag
\|\hat x-x\|&\le\|\hat x-\hat u\| +\|\hat u-x\| <\varepsilon+\|\hat u-x\|<\alpha d(x,[f\leq0]),
\\
f(\hat x)&>f(\hat u)-\varepsilon>0,
\\
f(\hat x)&<f(\hat u)+\varepsilon
\le\hat\tau d(\hat u,[f\le0])+\varepsilon
\\
&\le\hat\tau d(\hat u,[f\le0])+\varepsilon +\tau\big(d(\hat x,[f\le0])+\|\hat x-\hat u\|-d(\hat u,[f\le0])\big)
\\
&<(\hat\tau-\tau)d(\hat u,[f\le0])+\varepsilon(1+\tau) +\tau d(\hat x,[f\le0])<\tau d(\hat x,[f\le0]),
\\
f(\hat x)&<f(x)+\varepsilon<M,
\\
d(0,\partial f(\hat x))&\le\|x^*\|<
\hat\tau+\varepsilon<\tau.
\end{align*}
This contradicts \eqref{L1-1}.
\item
Suppose $f$ is convex.
Since $g$ is convex continuous, we can
apply the convex sum rule (Lemma~\ref{SR}(ii)) to find a subgradient $x^*\in\partial f(\hat u)$ such that  $\|x^*\|\le\hat\tau$.
Thus, making use also of \eqref{L1P-4}, we have $f(\hat u)<\tau d(\hat u,[f\le0])$ and $d(0,\partial f(\hat u))\le\|x^*\|\le\hat\tau<\tau$.
This contradicts \eqref{L1-1c} and completes the proof.
\qed\end{enumerate}
\end{proof}

In the setting of linear error bounds, the first part of Lemma~\ref{L1} strengthens \cite[Theorem~2]{NgaThe09}, where a more general setting of H\"older error bounds was studied.
We are going to show in the next subsection that this seemingly more general setting can be treated within
the conventional linear theory.

{Dropping or weakening any or all of the conditions on $u$ in \eqref{L1-1} makes the sufficient condition in Lemma~\ref{L1} stronger (while weakening the result).
This way one can formulate simplified versions of Lemma~\ref{L1}.
For instance, condition $f(u)<\tau d(u,[f\le0])$ in \eqref{L1-1} does not seem practical when checking error bounds as it involves the unknown set $[f\le0]$, and basically says that only the points not satisfying the error bound property with constant $\tau$ should be checked.
This condition is usually
{either dropped or replaced by the easier to check weaker condition $f(u)<\tau\|u-\bx\|$,
where $\bx$ is some point in $[f\le0]$.}

\begin{corollary}
Suppose $X$ is a Banach space, $f:X\rightarrow\mathbb{R}\cup\{+\infty\}$ is lower semicontinuous, $x\in X$ and $f(x)>0$.
Let $\tau>0$ and $\al\in(0,1]$.
If either $X$ is Asplund or $f$ is convex, and
\begin{gather*}
d(0,\partial f(u))\ge\tau
\;\;\mbox{for all}\;\;
u\in X
\;\mbox{with}\;
\|u-x\|<\alpha d(x,[f\leq0])
{\;\mbox{and}\;
f(u)<\tau d(u,[f\le0])},
\end{gather*}
then $[f\leq0]\ne\es$ and condition \eqref{L1-2} holds true.
\end{corollary}
}

\begin{remark}
\begin{enumerate}
\if{
\item
Instead of simply dropping condition $f(u)<\tau d(u,[f\le0])$ in \eqref{L1-1}, it can be replaced by the easier to check (but weaker) condition $f(u)<\tau\|u-\bx\|$.
This remark applies to the other statements in this paper deduced from Lemma~\ref{L1}: they can be simplified (though weakened!)
in a similar way.}\fi
\item
The subdifferential characterizations in Lemma~\ref{L1} are in fact consequen\-ces of the corresponding primal space characterizations in terms of slopes, some traces of which can be found in its proof; cf. \cite{Kru15,WuYe02}.
We do not consider primal space characterizations in this paper.
{
\item
Elementary (primal or dual) error bound statements for a fixed point $x\notin[f\le0]$, coming from the \EVP\ and lying at the core of all sufficient error bound characterizations have been formulated by several authors; cf. \cite{Iof79,Iof00,NgaThe09,Pen13}.}
\item
It is well understood by now that Fr\'echet subdifferentials can be replaced in this type of results by other subdifferentials possessing reasonable sum rules in appropriate (\emph{trustworthy} \cite{Iof98}) spaces.
For instance, it is easy to establish analogues of Lemma~\ref{L1} and the other statements in this section for \lsc\ functions in general Banach spaces in terms of Clarke subdifferentials.
{We do not do it in the current paper to keep the presentation simple and avoid using several types of subdifferentials in one statement.}
\item
The value of the parameter $\al$ in Lemma~\ref{L1} determines a tradeoff between the strength of the error bound estimate \eqref{L1-2} and the size of the neighbourhood of $x$ involved in the sufficient conditions \eqref{L1-1} and \eqref{L1-1c}: increasing the value of $\al$ strengthens condition \eqref{L1-2} at the expense of increasing the neighbourhood of $x$, all points from which have to be checked in conditions \eqref{L1-1} and \eqref{L1-1c}.
\end{enumerate}
\end{remark}

The next theorem is a slight generalization of the conventional linear error bound statement in the subdifferential form (which corresponds to taking $\al=1$).
It is an immediate consequence of Lemma~\ref{L1}.

\begin{theorem}\label{T3.1}
Suppose $X$ is a Banach space, $f:X\rightarrow\mathbb{R}\cup\{+\infty\}$ is lower semicontinuous and $\bar x\in[f\leq0]$.
Let $\tau>0$ and $\delta\in(0,\infty]$.
If either $X$ is Asplund or $f$ is convex, and
\begin{gather}\label{T3.1-1}
d(0,\partial f(x))\ge\tau
\quad\mbox{for all}\quad
x\in B_\delta(\bar x)\cap[f>0]
\;\;\mbox{with}\;\;
f(x)<\tau d(x,[f\leq0]),
\end{gather}
then
\begin{gather}\label{T3.1-2}
\alpha\tau d(x,[f\leq0])\leq f_+(x)
\quad\mbox{for all}\quad
\alpha\in(0,1]
\quad\mbox{and}\quad
x\in B_{\frac{\delta}{1+\alpha}}(\bar x).
\end{gather}
\end{theorem}

\begin{proof}
Suppose that condition \eqref{T3.1-2} is not satisfied, i.e.
$f_+(x)<\alpha\tau d(x,[f\leq0])$
for some $\alpha\in(0,1]$ and some
$x\in B_{\frac{\delta}{1+\alpha}}(\bar x)$.
Then $d(x,[f\leq0])>0$, and consequently, $f_+(x)=f(x)>0$.
By Lemma~\ref{L1}, there exists a $u\in X$
with
$\|u-x\|<\alpha d(x,[f\leq0])$
and
$f(u)<\tau d(u,[f\leq0])$ such that
$d(0,\partial f(u))<\tau$.
This contradicts \eqref{T3.1-1} because
$\|u-\bx\|\le\|u-x\|+\|x-\bx\|<(\al+1)\|x-\bx\|<\de$
and
$f(u)>0$.
\qed\end{proof}

{
The next statement is a simplified version of Theorem~\ref{T3.1}.
with the the last inequality in \eqref{T3.1-1}
{replaced by the easier to check weaker condition $f(u)<\tau\|u-\bx\|$}.

\begin{corollary}\label{C3.5}
Suppose $X$ is a Banach space, $f:X\rightarrow\mathbb{R}\cup\{+\infty\}$ is lower semicontinuous and $\bar x\in[f\leq0]$.
Let $\tau>0$ and $\delta\in(0,\infty]$.
If either $X$ is Asplund or $f$ is convex, and
\begin{gather*}
d(0,\partial f(x))\ge\tau
\quad\mbox{for all}\quad
x\in B_\delta(\bar x)\cap[f>0]
{\;\mbox{with}\;
f(x)<\tau\|x-\bx\|},
\end{gather*}
then
condition \eqref{T3.1-2} holds true.
\end{corollary}
}

\begin{remark}\label{R3.4}
\begin{enumerate}
\item
Theorem~\ref{T3.1}
{and Corollary~\ref{C3.5}}
allow for $\de=\infty$, thus, covering also global error bounds.
\item
The value of the parameter $\al$ in Theorem~\ref{T3.1} 
{and Corollary~\ref{C3.5}}
determines a tradeoff between the sharpness of the error bound estimate in \eqref{T3.1-2} and the size of the neighbourhood of $\bx$, where this estimate holds.
If the size of the neighbourhood is not important, one can take $\al=1$, which insures the sharpest error bound estimate.
Note that, unlike Lemma~\ref{L1}, in Theorem~\ref{T3.1} and the subsequent statements in this section the parameter $\al$ is only present in the concluding part. 
\if{
\item
The inequality $f(x)<\tau d(x,[f\leq0])$ in \eqref{T3.1-1} and similar conditions in the other error bound statements in this section can be replaced by the weaker but easier to check inequality $f(x)<\tau\|x-\bx\|$.}\fi
\end{enumerate}
\end{remark}

Thanks to
{Corollary~\ref{C3.5}},
the limit
\begin{gather*}
{\underline{\Er}f(\bx):=}
\liminf_{x\to\bar x,\;f(x)
{\downarrow}
0} d(0,\partial f(x))
\end{gather*}
provides a lower estimate for the local error bound modulus $\Er f(\bx)$ of $f$ at $\bx$.
{Such estimates are often used in the literature.}

\subsection{H\"older error bounds}

The estimate \eqref{T3.1-2} constitutes the \emph{linear} error bound for the function $f$ at $\bar x$ with
constant
$\al\tau$.
In many important situations such linear estimates do not hold, and this is where more subtle nonlinear (in particular, H\"older type) models come into play.
Surprisingly, such seemingly more general models can be treated within the conventional linear theory.
The next theorem providing a characterization for the H\"older error bounds is a consequence of Theorem~\ref{T3.1}.

Given a function $f:X\rightarrow\mathbb{R}\cup\{+\infty\}$, a point $x\in X$ with $f(x)\ge0$ and a number $q>0$, $f^q(x)$ stands for $[f(x)]^q$.
Thus,
$f^q$ is a function on $[f\ge0]$.
Note that the next theorem allows for $q>1$.

\begin{theorem}\label{T3.2}
Suppose $X$ is a Banach space, $f:X\rightarrow\mathbb{R}\cup\{+\infty\}$ is lower semicontinuous and $\bar x\in[f\leq0]$.
Let $\tau>0$, $\delta\in(0,\infty]$ and $q>0$.
If either $X$ is Asplund or $f$ is convex, and
\begin{gather}\label{T3.2-1}
q f^{q-1}(x)d(0,\partial f(x))\ge\tau
\;\;\mbox{for all}\;\;
x\in B_\delta(\bar x)\cap[f>0]
\;\mbox{with}\;
f^q(x)<\tau d(x,[f\leq0]),
\end{gather}
then
\begin{gather}\label{T3.2-2}
\alpha\tau d(x,[f\leq0])\leq f_+^q(x)
\quad\mbox{for all}\quad
\alpha\in(0,1]
\quad\mbox{and}\quad
x\in B_{\frac{\delta}{1+\alpha}}(\bar x).
\end{gather}
\end{theorem}
\begin{proof}
Apply Theorem~\ref{T3.1} with the lower semicontinuous function $x\mapsto f_+^q(x)$ in place of $f$.
Observe that $[f_+^q\leq 0]=[f_+=0]=[f\leq 0]$ and, for any $x\in[f>0]$, we have $f_+^q(x)=f^q(x)$ and $\partial f^q(x)=q f^{q-1}(x)\partial f(x)$ (by Lemma~\ref{L2}).
\qed\end{proof}

\if{
{Similarly, applying Proposition~\ref{P5} to the function $x\mapsto f_+^q(x)$ leads to the following statement.
\begin{proposition}\label{P9}
Suppose $X$ is a normed vector space, $f:X\rightarrow\mathbb{R}\cup\{+\infty\}$ and $\bar x\in[f\leq0]$.
Let numbers $\tau>0$, $\delta>0$ and $q\in(0,1]$ be given.
If $f_+^q$ is convex and
\begin{gather*}
\tau d(x,[f\leq0])\leq f_+^q(x)
\quad\mbox{for all}\quad
x\in B_\delta(\bar x),
\end{gather*}
then
\begin{gather*}
q f^{q-1}(x)d(0,\partial f(x))\ge\tau
\quad\mbox{for all}\quad
x\in B_\delta(\bar x)\cap[f>0].
\end{gather*}
\end{proposition}
}}\fi

Theorem~\ref{T3.2} strengthens \cite[Corollary~2, parts (i) and (ii)]{NgaThe09}.
When $q=1$, Theorem~\ref{T3.2} reduces to Theorem~\ref{T3.1}.


The next statement is a simplified version of Theorem~\ref{T3.2},
with the the last inequality in \eqref{T3.2-1}
{replaced by the easier to check weaker condition $f^q(u)<\tau\|u-\bx\|$}.

\begin{corollary}\label{C3.8}
Suppose $X$ is a Banach space, $f:X\rightarrow\mathbb{R}\cup\{+\infty\}$ is lower semicontinuous and $\bar x\in[f\leq0]$.
Let $\tau>0$, $\delta\in(0,\infty]$ and $q>0$.
If either $X$ is Asplund or $f$ is convex, and
\begin{gather}\label{P3.9-1}
q f^{q-1}(x)d(0,\partial f(x))\ge\tau
\;\;\mbox{for all}\;\;
x\in B_\delta(\bar x)\cap[f>0]
{\;\mbox{with}\;
f^q(x)<\tau\|x-\bx\|},
\end{gather}
then
condition \eqref{T3.2-2} holds true.
\end{corollary}

In view of
{Corollary~\ref{C3.8}} and definition \eqref{Er}, the limit
\begin{gather}\label{Er1}
\underline{\Erq}f(\bx):=q\liminf_{x\to\bar x,\;f(x)
{\downarrow}
0}\frac{d(0,\partial f(x))}{f^{1-q}(x)}
\end{gather}
provides a lower estimate for the modulus $\Erq f(\bx)$ of $q$-order error bounds of $f$ at $\bx$.
\if{
\begin{remark}
It is not difficult to show that condition $f(x)>0$ under the $\liminf$ in \eqref{Er1} and other similar limits in this paper can be replaced by $f(x)\downarrow0$.
\end{remark}
}\fi

\begin{example}\label{E3.6}
Let $f:\R\to\R$ be given by $f(x)=x^2$ if $x\ge0$ and $f(x)=0$ if $x<0$.
Then
{$[f\le0]=\R_-$ and},
for any $x>0$, we have $d(x,[f\le0])=x$, $d(0,\partial f(x))=f'(x)=2x$, and, with $q=\frac{1}{2}$, $q f^{q-1}(x)d(0,\partial f(x))=\frac{1}{2}\cdot\frac{1}{x}\cdot2x=1$.
Hence, condition \eqref{P3.9-1} is satisfied with $q=\frac{1}{2}$ and any $\tau\in(0,1]$ and $\delta\in(0,\infty]$.
With $q=\frac{1}{2}$, the inequality in \eqref{T3.2-2} becomes $\al\tau x_+\le x_+$, where $x_+:=\max\{x,0\}$.
It is indeed satisfied for all $\tau\in(0,1]$, $\al\in(0,1]$ and $x\in\R$.
\qed\end{example}

\begin{example}
Let $f:\R\to\R$ be given by $f(x)=\sqrt{x}$ if $x\ge0$ and $f(x)=0$ if $x<0$.
Then
{$[f\le0]=\R_-$ and},
for any $x>0$, we have $d(x,[f\le0])=x$, $d(0,\partial f(x))=f'(x)=\frac{1}{2\sqrt{x}}$, and, with $q=2$, $q f^{q-1}(x)d(0,\partial f(x))=2\cdot\sqrt{x}\cdot\frac{1}{2\sqrt{x}}=1$.
Hence, condition \eqref{P3.9-1} is satisfied with $q=2$ and any $\tau\in(0,1]$ and $\delta\in(0,\infty]$.
With $q=2$, the inequality in \eqref{T3.2-2} becomes $\al\tau x_+\le x_+$.
It is indeed satisfied for all $\tau\in(0,1]$, $\al\in(0,1]$ and $x\in\R$.
\qed\end{example}

It was observed in \cite{YaoZhe16} that applying the Ekeland variational principle in the proof of results like Theorem~\ref{T3.2} in a slightly different way, one can obtain a sufficient subdifferential condition for H\"older error bounds in a different form.
Next we show that conditions of this type are also
direct consequences of Lemma~\ref{L1}.
The following statement is motivated by \cite[Theorem~3.1]{YaoZhe16}.
Note that, unlike Theorem~\ref{T3.2}, it is restricted to the case $q\le1$.

\begin{theorem}\label{T3.3}
Suppose $X$ is a Banach space, $f:X\rightarrow\mathbb{R}\cup\{+\infty\}$ is lower semicontinuous and $\bar x\in[f\leq0]$.
Let $\tau>0$, $\delta\in(0,\infty]$ and $q\in(0,1]$.
If either $X$ is Asplund or $f$ is convex, and
\begin{multline}\label{T3.3-1}
d(x,[f\leq0])^{q-1}d(0,\partial f(x))^q\ge\tau
\quad\mbox{for all}\quad
x\in B_\delta(\bar x)\cap[f>0]\\
\mbox{with}\;\;
f^q(x)<\tau d(x,[f\leq0]),
\end{multline}
then
\begin{gather}\label{T3.3-2}
\alpha^q(1-\alpha)^{1-q}\tau d(x,[f\leq0])\leq f_+^q(x)
\quad\mbox{for all}\quad
\alpha\in(0,1)
\quad\mbox{and}\quad
x\in B_{\frac{\delta}{1+\alpha}}(\bar x).
\end{gather}
\end{theorem}

\begin{proof}
Suppose that condition \eqref{T3.3-2} is not satisfied, i.e.
{$f_+^q(x)<\alpha^q(1-\alpha)^{1-q}\tau d(x,{[f\leq0]})$}
for some $\al\in(0,1)$ and
$x\in B_{\frac{\delta}{1+\alpha}}(\bar x)$.
Then $d(x,[f\leq0])>0$, and consequently, $f_+(x)=f(x)>0$.
Set $\tau':=\tau^{
{\frac{1}{q}}
}\big((1-\alpha) d(x,[f\leq0])\big)^{\frac{1}{q}-1}$.
Then $0<f(x)<\alpha\tau' d(x,[f\leq0])$.
By Lemma~\ref{L1}, there exists a $u\in X$
with
$\|u-x\|<\alpha d(x,[f\leq0])$
and
$f(u)<\tau' d(u,[f\le0])$ such that $d(0,\partial f(u))<\tau'$.
Observe that
{$f(u)>0$,}
\begin{gather}\label{T3.3P1}
\|u-\bx\|\le\|u-x\|+\|x-\bx\|<(\al+1)\|x-\bx\|<\de,
\\\label{T3.3P2}
(1-\alpha) d(x,[f\leq0])<d(x,[f\leq0])-\|u-x\|\le d(u,[f\leq0]),
\\\label{T3.3P3}
{(\tau')^q=\tau
\big((1-\alpha) d(x,[f\leq0])\big)^{1-q} \overset{\eqref{T3.3P2}}{<}\tau d(u,[f\leq0])^{1-q}},
\\\label{T3.3P4}
{f^q(u)<(\tau')^q d(u,[f\le0])^q\overset{\eqref{T3.3P3}}{<}\tau d(u,[f\leq0])},
\\\notag
{d(u,[f\leq0])^{q-1}d(0,\partial f(u))^q <d(u,[f\leq0])^{q-1}(\tau')^q \overset{\eqref{T3.3P3}}{<}\tau}.
\end{gather}
In view of \eqref{T3.3P1} and \eqref{T3.3P4}, this contradicts \eqref{T3.3-1} and completes the proof.
\qed\end{proof}

Just like Theorem~\ref{T3.2}, when $q=1$ Theorem~\ref{T3.3} reduces to Theorem~\ref{T3.1}.
Thus, both Theorems~\ref{T3.2} and \ref{T3.3} generalize Theorem~\ref{T3.1} to the H\"older setting.

The next statement is a simplified version of Theorem~\ref{T3.3}.

\begin{corollary}\label{C3.12}
Suppose $X$ is a Banach space, $f:X\rightarrow\mathbb{R}\cup\{+\infty\}$ is lower semicontinuous and $\bar x\in[f\leq0]$.
Let $\tau>0$, $\delta\in(0,\infty]$ and $q\in(0,1]$.
If either $X$ is Asplund or $f$ is convex, and
\begin{multline}\label{C3.12-1}
d(x,[f\leq0])^{q-1}d(0,\partial f(x))^q\ge\tau
\quad\mbox{for all}\quad
x\in B_\delta(\bar x)\cap[f>0]
\\
{\mbox{with}\;
f^q(x)<\tau\|x-\bx\|},
\end{multline}
then
condition \eqref{T3.3-2} holds true.
\end{corollary}

As observed in Remark~\ref{R3.4}(ii) concerning Theorem~\ref{T3.1}
{and Corollary~\ref{C3.5}},
the value of the parameter $\al$ in Theorem~\ref{T3.3}
{and Corollary~\ref{C3.12}}
determines a tradeoff between the sharpness of the error bound estimate in \eqref{T3.3-2} and the size of the neighbourhood of $\bx$, where this estimate holds.
Thanks to the special form of the expression in the \LHS\ of the inequality in \eqref{T3.3-2}, the range of values of $\al$ in \eqref{T3.3-2} can be reduced, with the sharpest error bound estimate corresponding to taking $\al=q$.

\begin{proposition}\label{P3.12}
Under the assumptions of Theorem~\ref{T3.3}, and adopting the convention $0^0=1$, condition \eqref{T3.3-2} is equivalent to the following one:
\begin{gather}\label{T3.3-3}
\alpha^q(1-\alpha)^{1-q}\tau d(x,[f\leq0])\leq f_+^q(x)
\quad\mbox{for all}\quad
\alpha\in(0,q]
\quad\mbox{and}\quad
x\in B_{\frac{\delta}{1+\alpha}}(\bar x).
\end{gather}
The latter condition implies
\begin{gather}\label{T3.3-4}
q^q(1-q)^{1-q}\tau d(x,[f\leq0])\leq f_+^q(x)
\quad\mbox{for all}\quad
x\in B_{\frac{\delta}{1+q}}(\bar x).
\end{gather}
{Moreover, condition \eqref{T3.3-4} is equivalent to \eqref{T3.3-3} with the \nbh\ $B_{\frac{\delta}{1+\alpha}}(\bar x)$ replaced by $B_{\frac{\delta}{1+q}}(\bar x)$.}
\end{proposition}

\begin{proof}
The implication $\eqref{T3.3-3}\;\Rightarrow\;\eqref{T3.3-4}$ is obvious, as well as the implication $\eqref{T3.3-2}\;\Rightarrow\;\eqref{T3.3-3}$
when $q<1$.
It is easy to check that in the latter case the function $\alpha\mapsto\alpha^q(1-\alpha)^{1-q}$
is strictly increasing on $(0,q)$ and strictly decreasing on $(q,1)$.
Hence, when $\alpha>q$, one has
$\alpha^q(1-\alpha)^{1-q}<q^q(1-q)^{1-q}$
and $B_{\frac{\delta}{1+\alpha}}(\bar x)\subset B_{\frac{\delta}{1+q}}(\bar x)$, and consequently, $\eqref{T3.3-3}\;\Rightarrow\;\eqref{T3.3-2}$.

When $q=1$, the implication $\eqref{T3.3-3}\;\Rightarrow\;\eqref{T3.3-2}$ is obvious.
For the converse implication, only the case $\al=1$ needs to be covered.
Condition \eqref{T3.3-2} implies
\begin{gather*}
\alpha\tau d(x,[f\leq0])\leq f_+(x)
\quad\mbox{for all}\quad
\alpha\in(0,1)
\quad\mbox{and}\quad
x\in B_{\frac{\delta}{2}}(\bar x).
\end{gather*}
Taking supremum over $\al$ in the \LHS\ of the above inequality, we see that the inequality must hold also for $\al=1$.
{The `moreover' part is obvious since $\alpha^q(1-\alpha)^{1-q}\le q^q(1-q)^{1-q}$ for all $\al\in(0,q]$.}
\qed\end{proof}

Thanks to Proposition~\ref{P3.12}, the sufficient error bound condition in Corollary~\ref{C3.12} can be simplified further.
\begin{corollary}\label{C3.14}
Suppose $X$ is a Banach space, $f:X\rightarrow\mathbb{R}\cup\{+\infty\}$ is lower semicontinuous and $\bar x\in[f\leq0]$.
Let $\tau>0$, $\delta\in(0,\infty]$ and $q\in(0,1]$.
If either $X$ is Asplund or $f$ is convex, and
\begin{multline}\label{C3.14-1}
q^q(1-q)^{1-q}d(x,[f\leq0])^{q-1}d(0,\partial f(x))^q\ge\tau
\quad\mbox{for all}\quad
x\in B_\delta(\bar x)\cap[f>0]
\\
{\mbox{with}\;
q^q(1-q)^{1-q}f^q(x)<\tau\|x-\bx\|},
\end{multline}
then
\begin{gather}\label{C3.14-2}
\tau d(x,[f\leq0])\leq f_+^q(x)
\quad\mbox{for all}\quad
x\in B_{\frac{\delta}{1+q}}(\bar x).
\end{gather}
\end{corollary}

In view of Corollary~\ref{C3.8} and definition \eqref{Er}, the expression
\begin{gather}\label{Er2}
\underline{\Erq}'f(\bx):=q^q(1-q)^{1-q}\liminf_{x\to\bar x,\;f(x)
{\downarrow}
0} \frac{d(0,\partial f(x))^q}{d(x,[f\le0])^{1-q}}
\end{gather}
provides a lower estimate for the modulus $\Erq f(\bx)$ of $q$-order error bounds of $f$ at $\bx$ which complements \eqref{Er1}.

\begin{example}\label{E3.16}
Let $f:\R\to\R$ be defined as in Example~\ref{E3.6}: $f(x)=x^2$ if $x\ge0$ and $f(x)=0$ if $x<0$.
As computed in Example~\ref{E3.6}, for any $x>0$, we have $d(x,[f\le0])=x$ and $d(0,\partial f(x))=f'(x)=2x$.
Now, with $q=\frac{1}{2}$, we have for any $x>0$:
{
$$q^q(1-q)^{1-q}d(x,[f\leq0])^{q-1}d(0,\partial f(x))^q=\frac{1}{2}\cdot\frac{1}{\sqrt{x}}\cdot\sqrt{2x} =\frac{1}{\sqrt{2}}.$$
Hence, condition \eqref{C3.14-1} is satisfied with $q=\frac{1}{2}$ and any $\tau\in(0,\frac{1}{\sqrt{2}}]$ and $\delta\in(0,\infty]$.
Thus, Corollary~\ref{C3.14} gives in this example a global error bound estimate with constant up to $\frac{1}{\sqrt{2}}$,}
while we know from Example~\ref{E3.6}, which is a consequence of Theorem~\ref{T3.2}, that a global error bound estimate holds actually with any constant up to 1.
\qed\end{example}

The next proposition
shows that
{a slightly strengthened version of}
the sufficient error bound
{condition in Corollary~\ref{C3.8} implies that in Corollary~\ref{C3.12} (or \ref{C3.14}).}
In view of the obvious similarity of the concluding conditions \eqref{T3.2-2} (with $\al=1$) in Corollary~\ref{C3.8} and \eqref{C3.14-2} in Corollary~\ref{C3.14}, below we compare their assumptions.

\begin{proposition}\label{P3.9}
Suppose $X$ is a Banach space, $f:X\rightarrow\mathbb{R}\cup\{+\infty\}$ is lower semicontinuous and $\bar x\in[f\leq0]$.
Let $\tau>0$, $\delta\in(0,\infty]$ and $q\in(0,1]$,
and the convention $0^0=1$ be in force.
Suppose also that either $X$ is Asplund or $f$ is convex.
{If
\begin{gather}\label{P3.9-1+}
q f^{q-1}(x)d(0,\partial f(x))\ge\tau
\;\;\mbox{for all}\;\;
x\in B_\delta(\bar x)\cap[f>0]
{\;\mbox{with}\;
q^qf^q(x)<\tau\|x-\bx\|},
\end{gather}
then
\begin{gather*}
q^qd(x,[f\leq0])^{q-1}d(0,\partial f(x))^q\ge\tau
\;\;\mbox{for all}\;\;
x\in B_{\frac{\delta}{2}}(\bar x)\cap[f>0]
{\;\mbox{with}\;
q^qf^q(x)<\tau\|x-\bx\|},
\end{gather*}
i.e. condition \eqref{C3.14-1} is satisfied with $\tau':=(1-q)^{1-q}\tau$ and $\de':=\frac{\delta}{2}$ in place of $\tau$ and $\de$, respectively.}
\end{proposition}

\begin{proof}
Suppose that
condition \eqref{P3.9-1+} is satisfied.
Then,
{condition \eqref{P3.9-1} is satisfied too and,}
by
Corollary~\ref{C3.8},
condition \eqref{T3.2-2} holds true.
In view of conditions \eqref{P3.9-1} and \eqref{T3.2-2} with $\al=1$, we have for any $x\in B_{\frac{\delta}{2}}(\bar x)\cap[f>0]$
{with
$q^qf^q(x)<\tau\|x-\bx\|$}:
\begin{multline*}
q^qd(x,[f\leq0])^{q-1}d(0,\partial f(x))^q =d(x,[f\leq0])^{q-1}\big(q d(0,\partial f(x))\big)^q
\\
\ge d(x,[f\leq0])^{q-1}\big(\tau f^{1-q}(x)\big)^q
=\tau\left(\big(\tau d(x,[f\leq0])\big)^{-1}f^{q}(x)\right)^{1-q} \ge\tau.
\end{multline*}
This completes the proof.
\qed\end{proof}

Proposition~\ref{P3.9} allows us to establish a relationship between the lower error bound estimates \eqref{Er1} and \eqref{Er2}.

\begin{corollary}\label{C3.18}
Suppose $X$ is a Banach space, $f:X\rightarrow\mathbb{R}\cup\{+\infty\}$ is lower semicontinuous and $\bar x\in[f\leq0]$.
Let $q\in(0,1]$
and the convention $0^0=1$ be in force.
{
If $X$ is Asplund or $f$ is convex,
then
\begin{gather}\label{P3.9-3}
{(1-q)^{1-q}\,\underline{\Erq}f(\bx)\le \underline{\Erq}'f(\bx)}.
\end{gather}
}
\end{corollary}

\begin{proof}
If $\underline{\Erq}f(\bx)=0$, the first inequality in \eqref{P3.9-3} holds true trivially.
Suppose that $0<\tau<\underline{\Erq}f(\bx)$.
By definition \eqref{Er1}, condition \eqref{P3.9-1+} is satisfied with some number $\delta>0$ and, by Proposition~\ref{P3.9}, condition \eqref{C3.14-1} is satisfied with $\tau':=(1-q)^{1-q}\tau$ and $\de':=\frac{\delta}{2}$ in place of $\tau$ and $\de$, respectively.
Hence, by definition \eqref{Er2}, $\underline{\Erq}'f(\bx)\ge(1-q)^{1-q}\tau$.
Passing to the limit as $\tau\uparrow\underline{\Erq}f(\bx)$ proves
\eqref{P3.9-3}.
\qed\end{proof}

{
\begin{remark}
In view of
Corollary~\ref{C3.18}, the sufficient error bound
condition
$\underline{\Erq}'f(\bx)>0$
is in general weaker than
$\underline{\Erq}f(\bx)>0$.
At the same time, it also yields a weaker error bound estimate -- see \eqref{P3.9-3}.
This is illustrated by Example~\ref{E3.16}, where $\underline{\Erq}f(0)=1$, $\underline{\Erq}'f(0)=(1-q)^{1-q}=\frac{1}{\sqrt{2}}$, i.e. condition \eqref{P3.9-3} holds as equality.
\end{remark}
}

{
Inequality \eqref{P3.9-3} relating the two lower estimates for the modulus $\Erq f(\bx)$ can be strict.
Moreover, it can happen that $\underline{\Erq}f(\bx)=0$ while $\underline{\Erq}'f(\bx)>0$.
In such cases, $\underline{\Erq}'f(\bx)$ detects $q$-order error bounds while $\underline{\Erq}f(\bx)$ fails.
\begin{example}\label{E3.20}
Let $f:\R\to\R$ be given by
$$
f(x):=
\begin{cases}
0& \text{if } x\le0,
\\
x^2+\frac{1}{n}-\frac{1}{n^2}& \text{if } \frac{1}{n}<x\le\frac{1}{n-1},\; n=3,4\ldots,
\\
x^2+\frac{1}{4}& \text{if } x>\frac{1}{2}.
\end{cases}
$$
For any $n=3,4\ldots$ and $x\in\big(\frac{1}{n},\frac{1}{n-1}\big]$, we have $\frac{1}{n}<f(x) \le\frac{1}{(n-1)^2}+\frac{1}{n}-\frac{1}{n^2}$.
Hence $[f\le0]=\R_-$ and $f(x)\to0$ as $x\downarrow0$.
At the points $x_n:=\frac{1}{n-1}$, $n=3,4\ldots$,
the function is continuous from the left.
Moreover,
$$
f(x_n)-\lim_{x\downarrow x_n}f(x) =\frac{1}{(n-1)^2}+\frac{1}{n}-\frac{1}{n^2}-\frac{1}{n-1}
=-\frac{n^2-3n+1}{(n-1)^2n^2}<0.
$$
Hence, $f$ is \lsc.
For any $x>0$, we have $d(x,[f\le0])=x$, $f'(x)=2x$ if $x\ne x_n$ and $\sd f(x_n)=[2x_n,+\infty)$, $n=3,4\ldots$, and consequently,
$d(0,\partial f(x))=2x$ for all $x>0$.
With $q=\frac{1}{2}$, we have $d(x,[f\le0])^{q-1}d(0,\partial f(x))^q =x^{-\frac{1}{2}}(2x)^{\frac{1}{2}}=\sqrt{2}>0$ for any $x>0$;
hence, $\underline{\Erq}'f(\bx)>0$.
At the same time,
$f^{q-1}(x_n)d(0,\partial f(x_n)) <\left(\frac{1}{n}\right)^{-\frac{1}{2}}\frac{2}{n-1} =\frac{2\sqrt{n}}{n-1}\to0$ as $n\to\infty$;
hence, $\underline{\Erq}f(\bx)=0$.
\qed\end{example}
}

In some situations, it can be convenient to
{reformulate}
Theorem~\ref{T3.3} in a
{slightly}
different form given in the next corollary.

\begin{corollary}\label{C3.10}
Suppose $X$ is a Banach space, $f:X\rightarrow\mathbb{R}\cup\{+\infty\}$ is lower semicontinuous and $\bar x\in[f\leq0]$.
Let $\tau>0$, $\delta\in(0,\infty]$ and $p\ge0$.
If either $X$ is Asplund or $f$ is convex, and
\begin{multline}\label{C3.10-1}
d(0,\partial f(x))\ge\tau d(x,[f\leq0])^p
\quad\mbox{for all}\quad
x\in B_\delta(\bar x)\cap[f>0]\\
\mbox{with}\;\;
f(x)<\tau d(x,[f\leq0])^{p+1},
\end{multline}
then
\begin{gather}\label{C3.10-2}
\alpha(1-\alpha)^p\tau d(x,[f\leq0])^{p+1}\leq f_+(x)
\quad\mbox{for all}\quad
\alpha\in(0,1)
\quad\mbox{and}\quad
x\in B_{\frac{\delta}{1+\alpha}}(\bar x).
\end{gather}
\end{corollary}
\begin{proof}
Setting $q:=\frac{1}{p+1}$ and replacing $\tau$ with $\tau^{\frac{1}{p+1}}$ in the statement of Theorem~\ref{T3.3}, reduces it to that of the above corollary.
\qed\end{proof}

\begin{remark}
Corollary~\ref{C3.10} improves \cite[Corollary~3.1]{YaoZhe16}, which claims a weaker conclusion under stronger assumptions.
Condition \eqref{C3.10-2} is referred to in \cite{YaoZhe16} as \emph{$(p+1)$-order error bound}.
\end{remark}

Combining
{Corollaries~\ref{C3.8} and \ref{C3.14},} and
Proposition~\ref{P3.9},
we can formulate quantitative and qualitative sufficient subdifferential conditions for H\"older error bounds.

\begin{theorem}\label{T3.21}
Suppose $X$ is a Banach space, $f:X\rightarrow\mathbb{R}\cup\{+\infty\}$ is lower semicontinuous and $\bar x\in[f\leq0]$.
Let
{either $X$ is Asplund or $f$ is convex},
$\tau>0$, $q\in(0,1]$,
and the convention $0^0=1$ be in force.
Consider the following conditions:
\begin{enumerate}
\item
$\tau d(x,[f\leq0])\leq f_+^q(x)$
for all $x$ near $\bx$;
\item
$q f^{q-1}(x)d(0,\partial f(x))\ge\tau$
for all
$x\in[f>0]$ near $\bx$;
{
\item
$q^q(1-q)^{1-q}d(x,[f\leq0])^{q-1}d(0,\partial f(x))^q\ge\tau$
for all
$x\in[f>0]$ near $\bx$.}
\end{enumerate}
Then
{{\rm (ii) \folgt (i), (iii) \folgt (i)}, and {\rm (ii) \folgt (iii)} with $(1-q)^{1-q}\tau$ in place of $\tau$.}
If $q=1$, then conditions {\rm (ii)} and {\rm (iii)} coincide.
\end{theorem}

\if{
\begin{proof}
Implication (ii) \folgt (i) follows from Corollary~\ref{C3.8}.
Proposition~\ref{P3.9} yields implication (ii) \folgt (iii).
The last assertion is obvious.
\qed\end{proof}
}\fi

\begin{corollary}\label{C3.22}
Suppose $X$ is a Banach space, $f:X\rightarrow\mathbb{R}\cup\{+\infty\}$ is lower semicontinuous, $\bar x\in[f\leq0]$ and $q\in(0,1]$.
{Let either $X$ is Asplund or $f$ is convex.}
$f$ admits a $q$-order local error bound at $\bar x$ if one of the following
conditions is satisfied:
\begin{enumerate}
\item
$\liminf\limits_{x\to\bx,\,f(x)\downarrow0} f^{q-1}(x)d(0,\partial f(x))>0$;
\item
$\liminf\limits_{x\to\bx,\,f(x)\downarrow0} d(x,[f\leq0])^{q-1}d(0,\partial f(x))^q>0$.
\end{enumerate}
\end{corollary}

The last inequality in \eqref{C3.12-1} involving the $q$th power of the function $f$ can sometimes be replaced by a similar inequality involving the function $f$ itself.

\begin{proposition}\label{P3.16}
Suppose $X$ is  a Banach space, $f:X\rightarrow\mathbb{R}\cup\{+\infty\}$ is lower semicontinuous and $\bar x\in[f\leq0]$.
Let $\tau>0$, $\delta\in(0,\infty]$, $\beta>0$ and $q\in(0,1)$.
If either $X$ is Asplund or $f$ is convex, and
\begin{multline}\label{P3.16-1}
d(x,[f\leq0])^{q-1}d(0,\partial f(x))^q\ge\tau
\quad\mbox{for all}\quad
x\in B_\delta(\bar x)\cap[f>0]\\
\mbox{with}\;\;
f(x)<\beta \|x-\bar x\|,
\end{multline}
then, with $r:=\min\left\{\delta, \beta^{\frac{q}{1-q}}\tau^{-\frac{1}{1-q}}\right\}$,
\begin{gather}\label{P3.16-2}
\alpha^q(1-\alpha)^{1-q}\tau d(x,[f\leq0])\leq f_+^q(x)
\quad\mbox{for all}\quad
\alpha\in(0,1)
\quad\mbox{and}\quad
x\in B_{\frac{r}{1+\alpha}}(\bar x).
\end{gather}
\end{proposition}
\begin{proof}
If $x\in B_r(\bar x)\cap[f>0]$, then $x\in B_\delta(\bar x)\cap[f>0]$.
If, additionally, $f^q(x)<\tau
{\|x-\bx\|}
$, then
\begin{gather*}
f(x)<\tau^{\frac{1}{q}} \|x-\bar x\|^{\frac{1}{q}}<\tau^{\frac{1}{q}} r^{\frac{1}{q}-1}\|x-\bar x\|\le\be\|x-\bar x\|.
\end{gather*}
Hence, condition \eqref{P3.16-1} implies \eqref{C3.12-1}
with $r$ in place of $\delta$.
The statement follows from Theorem~\ref{T3.3}.
\qed\end{proof}

\begin{remark}
The above proposition is formulated for the case $q<1$.
When $q=1$, a similar assertion is trivially true with $\beta\ge\tau$ (as a consequence of Theorem~\ref{T3.1}), but, as the next example shows,
fails
when $\beta<\tau$.
This example shows also that
\cite[Proposition~3.1]{YaoZhe16} fails when $p=0$.
\end{remark}

\begin{example}
Let $f(x)=|x|$ $(x\in\mathbb{R})$, $\bar x=0$, $\tau=2$, $\beta=\frac{1}{2}$ and $q=1$.
Then there are no $x\in[f>0]$ with $f(x)<\beta|x|$, i.e. condition \eqref{P3.16-1} is trivially satisfied with any $\delta>0$.
Similarly, $\tau|x|>f(x)$ for any $x\ne0$, i.e. condition \eqref{P3.16-2} fails with any $r>0$.
\qed\end{example}

\begin{remark}
\begin{enumerate}
\item
One can easily formulate a statement similar to Proposition~\ref{P3.12} for the error bounds statements in Proposition~\ref{P3.16} and Corollary~\ref{C3.10}.
In the latter case, the sharpest error bound estimate in \eqref{C3.10-2} corresponds to taking $\al=\frac{1}{p+1}$, where the maximum of $\alpha(1-\alpha)^{p}$ over $\alpha\in(0,1)$ is attained.
\item
The neighbourhood $B_{\frac{\delta}{1+\alpha}}(\bar x)$ in \eqref{T3.3-2}, \eqref{T3.3-3} and \eqref{C3.10-2}, and the neighbourhood $B_{\frac{\delta}{1+q}}(\bar x)$ in \eqref{T3.3-4} and \eqref{C3.14-2} can always be replaced by the smaller neighbourhood $B_{\frac{\delta}{2}}(\bar x)$, independent of $\al$ or $q$.
A similar simplification is possible also in \eqref{P3.16-2}.
\item
Conditions \eqref{T3.3-1} in Theorem~\ref{T3.3}, \eqref{C3.12-1} in Corollary~\ref{C3.12}, \eqref{C3.10-1} in Corollary~\ref{C3.10}, (iii) in Theorem~\ref{T3.21}, (ii) in Corollary~\ref{C3.22} and \eqref{P3.16-1} in Proposition~\ref{P3.16}, although sufficient for the corresponding H\"older error bound estimates,
do not seem practical as they involve the unknown distance ${d(x,[f\leq0])}$, which error bounds are supposed to estimate.
This remark also applies to the next more general theorem.
Nevertheless, such conditions are in use in the literature; see \cite{ZheNg15,YaoZhe16,ZheZhu16}.
\end{enumerate}
\end{remark}

The next theorem combines the sufficient H\"older error bound conditions from Theorems~\ref{T3.2} and \ref{T3.3} in a single statement.
It is still a consequence of Lemma~\ref{L1}.

\begin{theorem}\label{T3.7}
Suppose $X$ is a Banach space, $f:X\rightarrow\mathbb{R}\cup\{+\infty\}$ is lower semicontinuous and $\bar x\in[f\leq0]$.
Let $\tau>0$, $\delta\in(0,\infty]$, $\la\in[0,1]$ and $q\in(0,1]$.
If either $X$ is Asplund or $f$ is convex, and
\begin{multline}\label{T3.7-1}
\left(\frac{\lambda} {d(x,[f\leq0])^{\frac{1}{q}-1}} +q(1-\lambda)f^{q-1}(x)\right)d(0,\partial f(x))\ge\tau
\qdtx{for all}
x\in B_\delta(\bar x)\cap[f>0]
\\
\mbox{with}\;\;
\frac{\lambda f(x)} {d(x,[f\leq0])^{\frac{1}{q}-1}} +(1-\lambda)f^q(x)<\tau d(x,[f\le0]),
\end{multline}
then
\begin{gather}\label{T3.7-2}
{\tau_\al d(x,[f\leq0])\leq f_+^q(x)}
\quad\mbox{for all}\quad
\alpha\in(0,1)
\quad\mbox{and}\quad
x\in B_{\frac{\delta}{1+\alpha}}(\bar x),
\end{gather}
where $\tau_\al>0$ is the unique solution for the equation
\begin{equation}\label{T3.7-3}
{\frac{\lambda \tau_\al^{\frac{1}{q}}} {(1-\alpha)^{\frac{1}{q}-1}}+(1-\lambda)\tau_\al} =\al\tau.
\end{equation}
\end{theorem}

\begin{proof}
Observe that the function $t\mapsto\varphi(t):=\lambda (1-\alpha)^{1-\frac{1}{q}}t^{\frac{1}{q}}+(1-\lambda)t$ is continuous and strictly increasing on $\R_+$ and satisfies $\varphi(0)=0$ and $\lim_{t\to\infty}\varphi(t)=\infty$.
Hence, the equation \eqref{T3.7-3} has a solution for any $\al>0$ and $\tau>0$, which is unique.
Suppose that condition \eqref{T3.7-2} is not satisfied, i.e.
\begin{equation}\label{T3.7P0}
{f_+^q(x)< \tau_\al d(x,[f\leq0])}
\end{equation}
for some $\al\in(0,1)$ and
$x\in B_{\frac{\delta}{1+\alpha}}(\bar x)$.
Then $d(x,[f\leq0])>0$, and consequently, $f_+(x)=f(x)>0$.
Consider a function $g:X\rightarrow\mathbb{R}\cup\{+\infty\}$ defined by
\begin{equation}\label{T3.7P1}
g(u):= \frac{\lambda f_+(u)} {\big((1-\alpha)d(x,[f\leq0])\big)^{\frac{1}{q}-1}} +(1-\lambda)f_+^q(u),\quad u\in X.
\end{equation}
It is obviously \lsc, $g(u)\ge0$ for all $u\in X$, $[g\le0]=[g=0]=[f\le0]$ and $g(x)>0$.
Observe that $g=\psi\circ f_+$, where $\psi(t):=\lambda \big((1-\alpha)d(x,[f\leq0])\big)^{1-\frac{1}{q}}t +(1-\lambda)t^q$, and $\psi:\R_+\to\R_+$ is strictly increasing and continuously differentiable on $(0,\infty)$.
Hence,
by \eqref{T3.7P0},
\begin{multline*}
g(x)=\psi(f(x))<\psi((\tau_\al d(x,[f\leq0]))^{\frac{1}{q}})\\
=\left(
{\frac{\lambda \tau_\al^{\frac{1}{q}}} {(1-\alpha)^{\frac{1}{q}-1}}+(1-\lambda)\tau_\al}
\right) d(x,[f\leq0])=\al\tau d(x,[g\leq0]).
\end{multline*}
Thus, $0<g(x)<\al\tau d(x,[g\leq0])$.
By Lemma~\ref{L1}, there exists a $u\in X$ such that
\begin{gather}\label{T3.7P2}
\|u-x\|<\alpha d(x,[f\leq0]),
\quad
g(u)<\tau d(u,[f\le0])
\qdtx{and}
d(0,\partial g(u))<\tau.
\end{gather}
{Hence, $f(u)>0$.}
The first inequality in \eqref{T3.7P2} immediately yields estimates \eqref{T3.3P1} and \eqref{T3.3P2}.
By \eqref{T3.7P1}
and the second inequality in \eqref{T3.7P2}, we have
\begin{gather}\label{T3.7P3}
\frac{\lambda f(u)} {d(u,[f\leq0])^{\frac{1}{q}-1}} +(1-\lambda)f^q(u)\le g(u)<\tau d(u,[f\le0]).
\end{gather}
Applying Lemma~\ref{L2},
we get
\begin{equation*}
\partial g(u)=\left(\frac{\lambda} {\big((1-\alpha)d(x,[f\leq0])\big)^{\frac{1}{q}-1}} +q(1-\lambda)f^{q-1}(u)\right) \partial f(u),
\end{equation*}
and consequently, by \eqref{T3.3P2} and the third inequality in \eqref{T3.7P2},
\begin{gather}\notag
\left(\frac{\lambda} {d(u,[f\leq0])^{\frac{1}{q}-1}} +q(1-\lambda)f^{q-1}(u)\right)d(0,\partial f(u))<d(0,\partial g(u))<\tau.
\end{gather}
In view of \eqref{T3.3P1} and \eqref{T3.7P3}, this contradicts \eqref{T3.7-1} and completes the proof.
\qed\end{proof}

\begin{remark}
When $\la=0$, Theorem~\ref{T3.7} reduces to Theorem~\ref{T3.2} except for the case $\al=1$ in \eqref{T3.2-2}.
When $\la=1$, Theorem~\ref{T3.7} reduces to Theorem~\ref{T3.3}.
When $q=1$, Theorem~\ref{T3.7} reduces to Theorem~\ref{T3.1} except for the case $\al=1$ in \eqref{T3.1-2}.
The case $\al=1$ in \eqref{T3.2-2} when $\la=0$ and in \eqref{T3.1-2} when $q=1$ is an immediate consequence of the case $\al\in(0,1)$; see the argument in the proof of Proposition~\ref{P3.12}.
\end{remark}

{
The next statement is a simplified version of Theorem~\ref{T3.7}.
\begin{corollary}
Suppose $X$ is a Banach space, $f:X\rightarrow\mathbb{R}\cup\{+\infty\}$ is lower semicontinuous and $\bar x\in[f\leq0]$.
Let $\tau>0$, $\delta\in(0,\infty]$, $\la\in[0,1]$ and $q\in(0,1]$.
If either $X$ is Asplund or $f$ is convex, and
\begin{multline*}
\left(\frac{\lambda} {d(x,[f\leq0])^{\frac{1}{q}-1}} +q(1-\lambda)f^{q-1}(x)\right)d(0,\partial f(x))\ge\tau
\qdtx{for all}
x\in B_\delta(\bar x)\cap[f>0]
\\
{\mbox{with}\;\;
\frac{\lambda f(x)} {d(x,[f\leq0])^{\frac{1}{q}-1}} +(1-\lambda)f^q(x)<\tau\|x-\bx\|},
\end{multline*}
then
condition \eqref{T3.7-2} holds true.
\end{corollary}
}
\subsection{Convex case}

{In this subsection $X$ is a normed vector space and the function $f:X\rightarrow\mathbb{R}\cup\{+\infty\}$ is assumed convex.
The statement of Lemma~\ref{L1} can be partially reversed (at the reference point).
\begin{lemma}\label{L3}
Suppose $X$ is a normed vector space, $f:X\rightarrow\mathbb{R}\cup\{+\infty\}$ is convex, $x\in X$, $f(x)>0$ and $\tau>0$.
If
\begin{gather}\label{L3-1}
\tau d(x,[f\leq0])\leq f(x),
\end{gather}
then $d(0,\partial f(x))\ge\tau$.
\end{lemma}
\begin{proof}
Let condition \eqref{L3-1} be satisfied and $x^*\in\sd f(x)$.
Then, for any $u\in[f\leq0]$, we have
$$
\|x^*\|\|u-x\|\ge-\ang{x^*,u-x}\ge f(x)-f(u)\ge f(x)\ge\tau d(x,[f\leq0]).
$$
Taking the infimum in the \LHS\ over all $u\in[f\leq0]$, we get $\|x^*\|\ge\tau$, which concludes the proof.
\qed\end{proof}
}

Combining Theorem~\ref{T3.1} and Lemma~\ref{L3},
we can formulate the standard subdifferential linear error bound criterion for convex functions.
\begin{theorem}
Suppose $X$ is a Banach space, $f:X\rightarrow\mathbb{R}\cup\{+\infty\}$ is convex lower semicontinuous, $\bar x\in[f\leq0]$ and $\tau>0$.
The following conditions are equivalent:
\begin{enumerate}
\item
$\tau d(x,[f\leq0])\leq f_+(x)$
for all $x$ near $\bx$;
\item
$d(0,\partial f(x))\ge\tau$
for all
$x\in[f>0]$ near $\bx$.
\end{enumerate}
\end{theorem}

The convex case `reverse' linear error bound statement in Lemma~\ref{L3} can also be easily adjusted to the H\"older setting both in the `conventional' form as in Theorem~\ref{T3.2} and its modification as in Theorem~\ref{T3.3}.
It is easy to see that the conclusion of the next lemma is actually a combination of two different conditions.

\begin{lemma}\label{L4}
Suppose $X$ is a normed vector space, $f:X\rightarrow\mathbb{R}\cup\{+\infty\}$ is convex, $x\in X$ and $f(x)>0$.
Let $\tau>0$ and $q\in(0,1]$.
If
\begin{gather}\label{L4-1}
\tau d(x,[f\leq0])\leq f^q(x),
\end{gather}
then $d(0,\partial f(x)) \ge\max\left\{\tau f^{1-q}(x),\tau^{\frac{1}{q}} d(x,[f\leq0])^{\frac{1}{q}-1}\right\}$,\\
or equivalently, $\min\left\{f^{q-1}(x)d(0,\partial f(x)), d(x,[f\leq0])^{q-1}d(0,\partial f(x))^q\right\}\ge\tau$.
\end{lemma}

\begin{proof}
Condition \eqref{L4-1} can be rewritten as
$$\tau f^{1-q}(x) d(x,[f\leq0])\leq f(x).$$
Applying Lemma~\ref{L3} with $\tau':=\tau f^{1-q}(x)$ in place of $\tau$, we get $d(0,\partial f(x))\ge\tau f^{1-q}(x)$.
Similarly, rewriting condition \eqref{L4-1} as
$$\tau^{\frac{1}{q}} d(x,[f\leq0])^{\frac{1}{q}-1} d(x,[f\leq0])\leq f(x),$$
and applying Lemma~\ref{L3} with $\tau':=\tau^{\frac{1}{q}} d(x,[f\leq0])^{\frac{1}{q}-1}$ in place of $\tau$, we get
$d(0,\partial f(x))\ge\tau^{\frac{1}{q}} d(x,[f\leq0])^{\frac{1}{q}-1}$.
\qed\end{proof}

Combining
Theorem~\ref{T3.21}
and Lemma~\ref{L4},
we can formulate quantitative and qualitative subdifferential characterizations of H\"older error bounds for convex functions.

\begin{theorem}
Suppose $X$ is a Banach space, $f:X\rightarrow\mathbb{R}\cup\{+\infty\}$ is convex lower semicontinuous and $\bar x\in[f\leq0]$.
Let $\tau>0$, $q\in(0,1]$,
and the convention $0^0=1$ be in force.
Consider
conditions
{{\rm (i), (ii)} and {\rm (iii)} in Theorem~{\rm \ref{T3.21}}.
Then
\begin{enumerate}
\item[\rm (a)]
{\rm (ii) \folgt (i)} and {\rm (i) \folgt (ii)} with $q\tau$ in place of $\tau$;
\item[\rm (b)]
{\rm (iii) \folgt (i)} and {\rm (i) \folgt (iii)} with $q^q(1-q)^{1-q}\tau$ in place of $\tau$.
\end{enumerate}
}
If $q=1$, then all the conditions are equivalent.
\end{theorem}

\if{
\begin{proof}
In addition to the implications coming from Theorem~\ref{T3.21}, we have an obvious implication
(vi) \folgt (v) and implications (i) \folgt (ii) and (i) \folgt (v) coming from Lemma~\ref{L4}.
If $q=1$, then conditions (ii)--(vii) coincide.
\qed\end{proof}
}\fi

\begin{corollary}\label{T3.25}
Suppose $X$ is a Banach space, $f:X\rightarrow\mathbb{R}\cup\{+\infty\}$ is convex lower semicontinuous, $\bar x\in[f\leq0]$ and $q\in(0,1]$.
The following conditions are equivalent:
\begin{enumerate}
\item
$\tau d(x,[f\leq0])\leq f_+^q(x)$
for some $\tau>0$ and all $x$ near $\bx$;
\item
$\liminf\limits_{x\to\bx,\,f(x)>0} f^{q-1}(x)d(0,\partial f(x))>0$;
\item
$\liminf\limits_{x\to\bx,\,f(x)>0} d(x,[f\leq0])^{q-1}d(0,\partial f(x))^q>0$.
\end{enumerate}
\end{corollary}

\section{Applications to convex semi-infinite optimization}\label{S5}

In this section, we mainly consider the following
convex optimization problem
\begin{equation}
\begin{aligned} P(c, b):\quad {\rm minimize}\quad \;\;&f(x)+\langle c,
x\rangle\\ {\rm subject\; to}\quad &g_t(x)\leq b_t, \;t\in T, \end{aligned}
\label{P}
\end{equation}%
where $c,\ x\in \mathbb{R}^{n}$, $T$ is a compact set in a metric space $Z$
such that\textbf{\ }$T\varsubsetneqq Z$, $f:\mathbb{R}^{n}\rightarrow
\mathbb{R}$ and $g_{t}:\mathbb{R}^{n}\rightarrow \mathbb{R},\ t\in T$, are
given convex functions
such
that $(t,x)\mapsto g_{t}(x)$ is
continuous on $T\times \mathbb{R}^{n}$, and $b\in \mathcal{C}(T,\mathbb{R})$%
, i.e., $T\ni t\mapsto b_{t}\in \mathbb{R}$ is continuous on $T$. In this
setting, the pair $(c,b)\in \mathbb{R}^{n}\times \mathcal{C}(T,\mathbb{R})$
is regarded as the parameter to be perturbed. The parameter space $\mathbb{R}%
^{n}\times \mathcal{C}(T,\mathbb{R})$ is endowed with the norm
\begin{equation}
\Vert (c,b)\Vert :=\max \{\Vert c\Vert ,\Vert b\Vert _{\infty }\},
\label{1.4}
\end{equation}%
where $\mathbb{R}^{n}$ is equipped with the Euclidean norm $\Vert \cdot
\Vert $ and $\Vert b\Vert _{\infty }:=\max_{t\in T}|b_{t}|$.

Our aim here is to analyze the \emph{solution mapping} (also called \emph{%
argmin mapping}) of problem \eqref{P}:
\begin{equation*}
\mathcal{S}:(c,b)\mapsto \{x\in \mathbb{R}^{n}\mid x\;\mathrm{solves}%
\;P(c,b)\}\;\mathrm{with}\;(c,b)\in \mathbb{R}^{n}\times \mathcal{C}(T,%
\mathbb{R}).
\end{equation*}%
In the special case that $c$ is fixed, $\mathcal{S}$ reduces to the partial
solution mapping $\mathcal{S}_{c}:\mathcal{C}(T,\mathbb{R})\rightrightarrows
\mathbb{R}^{n}$ given by
\begin{equation*}
\mathcal{S}_{c}(b)=\mathcal{S}(c,b).
\end{equation*}%
Associated with the parameterized problem $P(c,b)$, we denote by $\mathcal{F}
$ the feasible set mapping, which is given by
\begin{equation*}
\mathcal{F}(b):=\{x\in \mathbb{R}^{n}\mid g_{t}(x)\leq b_{t},t\in T\}.
\end{equation*}%
The set of \emph{active indices} at $x\in \mathcal{F}(b)$ is the set $%
T_{b}(x)$ defined by
\begin{equation*}
T_{b}(x):=\{t\in T\mid g_{t}(x)=b_{t}\}.
\end{equation*}

We say that the problem $P(c,b)$ satisfies the \emph{Slater constraint
qualification} (hereinafter called the \emph{Slater condition})\emph{\ }if
there exists $\hat{x}\in \mathbb{R}^{n}$ such that $g_{t}(\hat{x})<b_{t}$
for all $t\in T$.
The following well-known result (see \cite[Theorems 7.8 and 7.9]{GobLop98}) plays a key role in our
analysis.

\begin{proposition}
Let $(\bar{c},\bar{b})\in \mathbb{R}^{n}\times \mathcal{C}(T,\mathbb{R})$
and assume that $P(\bar{c},\bar{b})$ satisfies the Slater condition. Then $%
\bar{x}\in \mathcal{S}(\bar{c},\bar{b})$ if and only if the \emph{%
Karush-Kuhn-Tucker (KKT)} conditions hold, i.e.,
\begin{equation*}
\bar{x}\in \mathcal{F}(\bar{b})\quad \text{and}\quad -(\partial f(\bar{x})+%
\bar{c})\bigcap \left( \mathrm{cone}\Big(\bigcup_{t\in T_{\bar{b}}(\bar{x}%
)}\partial g_{t}(\overline{x})\Big)\right) \neq \emptyset .
\end{equation*}
\end{proposition}

Here $\mathrm{cone}(X)$ represents the conical convex hull of $X$, and we
assume that $\mathrm{cone}(X)$ always contain the zero-vector $0_{n}$, in
particular $\mathrm{cone}(\emptyset )=\{0_{n}\}.$

In this section we provide a characterization for H\"{o}lder calmness of $%
\mathcal{S}$ at $((\bar{c},\bar{b}),\bar{x})$. To this aim, we use the
following \emph{level set mapping} $\mathcal{L}:\mathbb{R}\times C(T,\mathbb{%
R})\rightrightarrows \mathbb{R}^{n}$ given by
\begin{equation*}
\mathcal{L}(\alpha ,b):=\{x\in \mathbb{R}^{n}\mid \;f(x)+\langle \bar{c}%
,x\rangle \leq \alpha ;\;g_{t}(x)\leq b_{t},t\in T\}
\end{equation*}%
and the \emph{supremum function} $\bar{f}:\mathbb{R}^{n}\rightarrow \mathbb{R%
}$ defined as
\begin{align}
\bar{f}(x):=&\sup \{f(x)-f(\bar{x})+\langle \bar{c},x-\bar{x}\rangle
;\;g_{t}(x)-\bar{b}_{t},\;t\in T\}  \label{2.5} \\
=&\sup \{f(x)+\langle \bar{c},x\rangle -\left( f(\bar{x})+\langle \bar{c},%
\bar{x}\rangle \right) ;\;g_{t}(x)-\bar{b}_{t},\;t\in T\}.  \notag
\end{align}%
(See \cite[(11) and (12)]{CanKruLopParThe14} for the linear counterparts of $\mathcal{L}$
and $\bar{f}$.)

For a given $t_{0}\in Z\diagdown T,$ we define\textbf{\ }%
\begin{equation*}
\overline{T}:=T\cup \{t_{0}\},\ g_{t_{0}}(x):=f(x)+\langle \bar{c},x\rangle
\text{ and }\bar{b}_{t_{0}}:=f(\bar{x})+\langle \bar{c},\bar{x}\rangle .
\end{equation*}%
As $\overline{T}$ is a compact set (and $t_{0}$ is an isolated point in $%
\overline{T}$), the function $(t,x)\mapsto g_{t}(x)$ is continuous on $%
\overline{T}\times \mathbb{R}^{n}$, $b\in \mathcal{C}(\overline{T},\mathbb{R)%
}$ and, obviously,
\begin{equation*}
\bar{f}(x)=\sup \{g_{t}(x)-\bar{b}_{t},\;t\in \overline{T}\}.
\end{equation*}

For any $x\in \mathbb{R}^{n}$, we consider the extended active set
\begin{equation*}
\overline{T}(x):=\{t\in \overline{T}:\;g_{t}(x)-\bar{b}_{t}=\bar{f}(x)\}.
\end{equation*}%
The following well-known result is useful for us (e.g.
\cite[VI, Theorem 4.4.2]{HirLem93}).
\begin{equation}
\partial\bar{f}(x)=\mathrm{co}\Big(\mathop{\bigcup}\limits_{t\in\overline
{T}(x)}\partial g_{t}(x)\Big).\label{marco3}
\end{equation}
Observe that
\begin{equation*}
\partial g_{t_{0}}(x)=\partial f(x)+\bar{c}.
\end{equation*}%
Since $((\bar{c},\bar{b}),\bar{x})\in\mathrm{gph}%
(\mathcal{S})$,
\begin{equation}
\mathcal{S}(\bar{c},\bar{b})=\left[  \bar{f}=0\right]  =\left[  \bar{f}%
\leq0\right]  =\mathcal{L}(f(\bar{x})+\langle\bar{c},\bar{x}\rangle,\bar
{b}).\label{marco6}%
\end{equation}
Observe that $t_{0}\in\overline{T}(\bar{x}).$
Consequently $0_{n}\in \partial \overline{f}(\overline{x})$, and by (\ref%
{marco3})
\begin{equation*}
0_{n}=\sum\nolimits_{i=1}^{p}\lambda _{i}u^{i},
\end{equation*}%
with $u^{i}\in \partial g_{t_{i}}(\overline{x}),$ $\{t_{i},\
i=1,2,\ldots,p\}\subset \overline{T}(\overline{x}),$ $\lambda _{i}>0$ and $%
\sum\nolimits_{i=1}^{p}\lambda _{i}=1.$

If $P(\bar{c},\bar{b})$ satisfies the Slater condition,
$t_{0}$ must be one of the indices involved in the sum above, and we shall
write
\begin{equation}
0_{n}=\mu_{0}(u^{0}+\overline{c})+\sum\nolimits_{i=1}^{q}\mu_{i}%
u^{i},\label{marco4}%
\end{equation}
with $u^{0}\in\partial f(\overline{x}),$ $u^{i}\in\partial g_{t_{i}}%
(\overline{x}),$ $\{t_{i},\ i=1,2,\ldots,q\}\subset T_{\overline{b}}(\overline
{x}),$ $\mu_{0}>0,$ $\mu_{i}\geq0,\ i=1,2,\ldots,q,$ and $\sum\nolimits_{i=0}%
^{q}\mu_{i}=1.$ Otherwise $0_{n}\in\mathrm{co}\left(  \bigcup\nolimits_{t\in
T_{\overline{b}}(\overline{x})}\partial g_{t}(\overline{x})\right)  $ and
$\overline{x}$ would be a global minimum of the function
\[
\varphi(\cdot):=\sup\{g_{t}(\cdot)-\bar{b}_{t},\;t\in T\},
\]
giving rise to the contradiction $\varphi(\widehat{x})<0=\varphi(\overline
{x}).$ Observe that it may happen that $\mu_{0}=1$ and the sum in
(\ref{marco4}) vanishes (this is the case if $T_{\overline{b}}(\overline
{x})=\emptyset$).

The following lemma provides a uniform boundedness result which is needed
later. It constitutes a convex counterpart of \cite[Lemma 3.2]{CanHanParTol14}.

\begin{lemma}
\label{L1.2} Let $((\bar{c},\bar{b}),\bar{x})\in \mathrm{gph}(\mathcal{S})$
be given and assume that $P(\bar{c},\bar{b})$ satisfies the Slater
condition. Then there exist $M>0$ and neighborhoods $U$ of $\bar{x}$ and $V$
of $(\bar{c},\bar{b})$ such that, for all $(c,b)\in V$ and all $x\in \mathcal{%
S}(c,b)\cap U$, there exists $u\in \partial f(x)$ satisfying
\begin{equation}
-(c+u)\in \lbrack 0,M]\ \mathrm{co}\Big(\mathop{\bigcup}\limits_{t\in
T_{b}(x)}\partial g_{t}(x)\Big).  \label{1.25}
\end{equation}
\end{lemma}

\begin{proof}
The result follows arguing by contradiction. By continuity we
can assume that $P(c,b)$ satisfies the Slater condition at any $(c,b)\in V$.
Then, thanks to the KKT conditions, together with Carath\'{e}odory Theorem,
there would exist a sequence $\mathrm{gph}(\mathcal{S})\ni
((c^{r},b^{r}),x^{r})\rightarrow ((\bar{c},\bar{b}),\bar{x})$ such that
\begin{equation}
-(c^{r}+u^{r})=\sum_{i=1}^{n}\lambda _{i}^{r}u_{i}^{r}  \label{marco5}
\end{equation}%
for some $\{t_{1}^{r},\dots ,t_{n}^{r}\}\subset T_{b^{r}}(x^{r})$, $\lambda
_{i}^{r}\geq 0$, and some $u^{r}\in \partial f(x^{r}),\ u_{i}^{r}\in
\partial g_{t_{i}}(x^{r}),\ r\in \mathbb{N}$, verifying $\sigma
_{r}:=\sum_{i=1}^{n}\lambda _{i}^{r}\rightarrow +\infty $ as $r$ tends to $%
+\infty $.

If we apply a filtering process as in \cite[Lemma~3.1]{CanHanLopPar08}, based on
the compactness of $T$ and \cite[Theorem~24.5]{Roc70}, we get the
existence of points $\{t_{1},\dots ,t_{n}\}\subset T_{\overline{b}}(%
\overline{x})$ such that $t_{i}^{r}\rightarrow t_{i},\ i=1,2,\ldots,n,$ and $%
u_{i}\in \partial g_{t_{i}}(\overline{x}),$ $i=1,2,\ldots,n,$ such that $%
u_{i}^{r}\rightarrow u_{i}$, $i=1,2,\ldots,n.$ Then, dividing both terms in (\ref%
{marco5}) by $\sigma _{r}$ and taking limits
as
$r\rightarrow \infty $
(after filtering again
with
respect to the bounded coefficients $\lambda
_{i}^{r}/\sigma _{r}$, $i=1,2,\ldots,n)$, we reach the same contradiction with the
Slater condition.
\qed\end{proof}

The following proposition gives a characterization of the $q$-order calmness property for the level set mapping
$%
\mathcal{L}$ in terms of the supremum function defined by \eqref{2.5}. It
constitutes a H\"{o}lder convex counterpart of \cite[Proposition 3.1]%
{CanHanParTol14} (see also \cite[Theorem 4]{CanKruLopParThe14}).
Recall
that $q\in
(0,1]$.

\begin{proposition}
\label{P1} Let $((\bar c, \bar b), \bar x)\in \mathrm{gph}(\mathcal{S})$.
Then the following are equivalent:\newline
(i) $\mathcal{L}$ is $q$-order calm at $((f(\bar x)+\langle \bar c, \bar
x\rangle, \bar b), \bar x)\in \mathrm{gph}(\mathcal{L})$;\newline
(ii)$\liminf\limits_{x\to \bar x, \bar f(x)\downarrow0}\bar f(x)^{q-1}d(0,
\partial \bar f(x))>0$.
\end{proposition}

\begin{proof}
This result is a direct consequence
of the equivalence between the $q$-order calmness of $\mathcal{L}$ at
$((f(\bar{x})+\langle\bar{c},\bar{x}\rangle,\bar{b}),\bar{x})$ and the the
existence of a $q$-order error bound of $\bar{f}$ at $\bar{x},$ together with
{Corollary}~\ref{T3.25},
(\ref{marco6}), and the following inequalities:
\begin{eqnarray}
\left[ \bar{f}(x)\right] _{+} &=&\left[ \sup \{f(x)-f(\bar{x})+\langle \bar{c%
},x-\bar{x}\rangle ;\;g_{t}(x)-\bar{b}_{t},\;t\in T\}\right] _{+}  \notag \\
&=&\sup \{[f(x)-f(\bar{x})+\langle \bar{c},x-\bar{x}\rangle
]_{+};\;[g_{t}(x)-\bar{b}_{t}]_{+},\;t\in T\}  \label{marco13} \\
&=&d((f(\bar{x})+\langle \bar{c},\bar{x}\rangle ,\bar{b}),\mathcal{L}%
^{-1}(x))\;\mathrm{for\;all}\;x\in \mathbb{R}^{n}.  \notag
\end{eqnarray}%
The proof is complete.
\qed\end{proof}

\begin{proposition}
Assume that $\mathcal{L}$ is not $q$-order calm at $((f(\bar{x}%
)+\langle\bar{c},\bar{x}\rangle,\bar{b}),\bar{x})\in\mathrm{gph}%
(\mathcal{L}).$ Then, there will exist sequences $\{x^{r}\}_{r\in\mathbb{N}}$
converging to $\bar{x}$ with $\bar{f}(x^{r})\downarrow0$, and $\{v^{r}%
\}_{r\in\mathbb{N}},$ $v^{r}\in\partial\overline{f}(x^{r})\diagdown\{0_{n}\},$
such that $v^{r}$ converges to $0_{n}$, and
\begin{equation}
d(x^{r},\mathcal{S}(\bar{c},\bar{b}))\geq\frac{\bar{f}(x^{r})}{\Vert
v^{r}\Vert}.\label{214}%
\end{equation}
\end{proposition}

\begin{proof}
Certainly, if
\[
\liminf\limits_{x\rightarrow\bar{x},\bar{f}(x)\downarrow0}\bar{f}%
(x)^{q-1}d(0,\partial\bar{f}(x))=0,
\]
there must exist sequences $x^{r}\rightarrow\bar{x},$ with $\bar{f}%
(x^{r})\downarrow0,$ and $v^{r}\in\partial\overline{f}(x^{r}),$ such that
\begin{equation}
\lim\limits_{r\rightarrow+\infty}\bar{f}(x^{r})^{q-1}v^{r}=0_{n}%
,\label{marco2}%
\end{equation}
entailing $v^{r}\rightarrow0_{n}$ as
\[
\lim_{r\rightarrow+\infty}\bar{f}(x^{r})^{q-1}=\left\{
\begin{array}
[c]{ll}%
+\infty, & \text{if }q<1,\\
1, & \text{if }q=1.
\end{array}
\right.
\]

First, we observe that $v^{r}\neq0_{n},\quad r=1,2,...$ Otherwise, i.e. if
$v^{r}=0_{n}$ for some $r$, then $0_{n}\in\partial\overline{f}(x^{r})$ and
$x^{r}$ is a (global) minimum of the convex function $\overline{f}$, entailing
$\overline{f}(x^{r})\leq\overline{f}(\overline{x})=0,$ but this contradicts
$\overline{f}(x^{r})>0.$

Finally, (\ref{214}) follows from the obvious fact
\[
\mathcal{S}(\bar{c},\bar{b})=\{x\in\mathbb{R}^{n}|\;\bar{f}(x)=0\}=\{x\in
\mathbb{R}^{n}|\;0_{n}\in\partial\bar{f}(x)\}.
\]
Since $v^{r}\in\partial\bar{f}(x^{r}),$ we have
\[
x\in\mathcal{S}(\bar{c},\bar{b})\Rightarrow0=\bar{f}(x)\geq\bar{f}%
(x^{r})+\left\langle v^{r},x-x^{r}\right\rangle ,
\]
and we conclude
\[
\mathcal{S}(\bar{c},\bar{b})\subset\{x\in\mathbb{R}^{n}|\;\left\langle
v^{r},x\right\rangle \leq\left\langle v^{r},x^{r}\right\rangle -\bar{f}%
(x^{r})\}.
\]
Applying the well-known Ascoli formula for the distance to a hyperplane we
get
\[
d(x^{r},S(\bar{c},\bar{b}))\geq\frac{\lbrack\bar{f}(x^{r})]_{+}}{\Vert
v^{r}\Vert}=\frac{\bar{f}(x^{r})}{\Vert v^{r}\Vert}.
\]
\qed\end{proof}

The following proposition provides a necessary condition in the case that
$\mathcal{L}$ is not $q$-order calm. The proof updates some arguments in
\cite[Theorem 3.1]{CanHanParTol14} to the convex $q$-H\"{o}lder setting.

\begin{proposition}\label{P3.2}
Let $((\bar{c},\bar{b}),\bar{x})\in\mathrm{gph}%
(\mathcal{S})$ and assume that $P(\bar{c},\bar{b})$ satisfies the Slater
condition. Suppose that $\mathcal{L}$ is not $q$-order calm at $((f(\bar
{x})+\langle\bar{c},\bar{x}\rangle,\bar{b}),\bar{x})\in\mathrm{gph}%
(\mathcal{L})$. Then there exist sequences $\{x^{r}\}_{r\in\mathbb{N}}$
converging to $\bar{x}$ and $\{b^{r}\}_{r\in\mathbb{N}}\subset C(T,\mathbb{R}%
)$ converging to $\bar{b}$ such that
\begin{equation}
x^{r}\in\mathcal{F}(b^{r}),\quad\lim_{r\rightarrow+\infty}\frac{\Vert
b^{r}-\bar{b}\Vert_{\infty}^{q}}{d(x^{r},S_{\bar{c}}(\bar{b}))}=0,\label{1.41}%
\end{equation}
as well as a finite set $T_{0}\subset T_{b^{r}}(x^{r})$ satisfying
\begin{equation}
-(\bar{c}+u)\in\sum_{t\in T_{b^{r}}(x^{r})}\gamma_{t}u_{t}\label{3.17}%
\end{equation}
with $\gamma _{t}>0,\ u_{t}\in \partial g_{t}(\bar{x}),\ t\in T_{0},$ and $%
u\in \partial f(\bar{x})$.
\end{proposition}

\begin{proof}
We have established that there exist sequences $\{x^{r}%
\}_{r\in\mathbb{N}}$ converging to $\bar{x}$ with $\bar{f}(x^{r})\downarrow0$,
and $\{v^{r}\}_{r\in\mathbb{N}},$ $v^{r}\in\partial\overline{f}(x^{r}%
)\diagdown\{0_{n}\},$ such that $v^{r}\rightarrow0$ and (\ref{214}) and
(\ref{marco2}) hold.

Applying Proposition 5.1 (remember that $P(\bar{c},\bar{b})$ satisfies the
Slater condition), we know that, associated with $\bar{x}\in S(\bar{c},\bar
{b}),$ there is a finite subset $T_{0}\subset T_{\bar{b}}(\bar{x})$ such that
\begin{equation}
-(\bar{c}+u)=\sum_{t\in T_{0}}\gamma_{t}u_{t},\label{1.44}%
\end{equation}
for some $\gamma_{t}>0,\ u_{t}\in\partial g_{t}(\bar{x}),\ t\in T_{0},$ and
$u\in\partial f(\bar{x}).$ Now we proceed by showing the existence of $N>0$
such that
\begin{equation}
g_{t}(x^{r})-\bar{b}_{t}\geq-N\bar{f}(x^{r})\quad\forall t\in T_{0}%
\;\text{and}\;r\in\mathbb{N}.\label{1.45}%
\end{equation}
We have that (\ref{1.44}) gives rise to
\begin{align}
-\sum_{t\in T_{0}}\gamma_{t}(g_{t}(x^{r})-\bar{b}_{t})  & =-\sum_{t\in T_{0}%
}\gamma_{t}(g_{t}(x^{r})-g_{t}(\bar{x}))\leq-\sum_{t\in T_{0}}\gamma
_{t}\langle u_{t},x^{r}-\bar{x}\rangle\nonumber\\
& =\langle\bar{c}+u,x^{r}-\bar{x}\rangle\leq\langle\bar{c},x^{r}-\bar
{x}\rangle+f(x^{r})-f(\bar{x})\nonumber\\
& =\bar{f}(x^{r}).\label{1.46}%
\end{align}
The set $T_{0}$ is finite, and this allows us to suppose that the following
sets are independent of $r$ (by taking a suitable subsequence if needed);
\[
T_{0}^{-}=\{t\in T_{0}|\;g_{t}(x^{r})-\bar{b}_{t}<0\}\quad\text{and}\quad
T_{0}^{+}=\{t\in T_{0}|\;g_{t}(x^{r})-\bar{b}_{t}\geq0\}.
\]
The inequality \eqref{1.45} is obvious for $t\in T_{0}^{+}$. In the
non-trivial case, i.e. when $T_{0}^{-}\neq\emptyset$, for any $\widetilde
{t}\in T_{0}^{-}$ we could deduce from \eqref{1.46} and the definition of
$\bar{f}(x^{r})$ that
\begin{align}
-\gamma_{\widetilde{t}}(g_{\widetilde{t}}(x^{r})-\bar{b}_{\widetilde{t}}) &
\leq-\sum_{t\in T_{0}^{-}}\gamma_{t}(g_{t}(x^{r})-\bar{b}_{t})\leq\bar
{f}(x^{r})+\sum_{t\in T_{0}^{+}}\gamma_{t}(g_{t}(x^{r})-\bar{b}_{t}%
)\label{marco12}\\
&  \leq\bar{f}(x^{r})+\sum_{t\in T_{0}^{+}}\gamma_{t}\bar{f}(x^{r}),\nonumber
\end{align}
and this implies that
\[
g_{\widetilde{t}}(x^{r})-\bar{b}_{\widetilde{t}}\geq\frac{\left(
1+\sum\limits_{t\in T_{0}^{+}}\gamma_{t}\right)  \bar{f}(x^{r})}%
{-\gamma_{\widetilde{t}}}.
\]
Accordingly we take
\[
N:=\max_{t\in T_{0}^{-}}\frac{1+\sum_{t\in T_{0}^{+}}\gamma_{t}}%
{\gamma_{\widetilde{t}}},
\]
which satisfies (\ref{1.45}).

Next we build the sequence $\{b^{r}\}_{r\in\mathbb{N}}$. Urysohn's Lemma
yields the existence, for each $r$, of a function $\varphi_{r}\in C(T,[0,1])$
such that
\[
\varphi_{r}(t)=\left\{
\begin{array}
[c]{ll}%
0, & \ \text{if}\;g_{t}(x^{r})-\bar{b}_{t}\geq-N\bar{f}(x^{r}),\\
1, & \ \text{if }g_{t}(x^{r})-\bar{b}_{t}\leq-(N+1)\bar{f}(x^{r}),
\end{array}
\right.
\]

Then, for each $t\in T$, we define
\begin{equation}
b_{t}^{r}:=(1-\varphi_{r}(t))g_{t}(x^{r})+\varphi_{r}(t)(\bar{b}_{t}+\bar
{f}(x^{r})).\label{4.26}%
\end{equation}
If the set\textbf{ }$\{t\in T\mid g_{t}(x^{r})-\bar{b}_{t}\leq-(N+1)\bar
{f}(x^{r})\}$ is empty we take\textbf{ }$\varphi_{r}(t)=0$ for all $t\in T.$

For each $r$,
\[
b_{t}^{r}-g_{t}(x^{r})=\varphi_{r}(t)(\bar{b}_{t}+\bar{f}(x^{r})-g_{t}%
(x^{r}))\geq0,
\]
and, thus, $x^{r}\in\mathcal{F}(b^{r}).$

We easily check that, when $\varphi_{r}(t)=1,$
\[
b_{t}^{r}-\bar{b}_{t}=\bar{f}(x^{r})<(N+1)\bar{f}(x^{r}),
\]
and when $\varphi_{r}(t)<1,$ $-(N+1)\bar{f}(x^{r})<g_{t}(x^{r})-\bar{b}%
_{t}\leq\bar{f}(x^{r}),$ entailing%
\[
b_{t}^{r}-\bar{b}_{t}\geq-(N+1)\bar{f}(x^{r}).
\]
Therefore, for all $t\in T$ and all $r\in\mathbb{N}$,
\begin{equation}
|b_{t}^{r}-\bar{b}_{t}|\ \leq(N+1)\bar{f}(x^{r}).\label{1.51}%
\end{equation}
In addition, \eqref{1.45} yields $T_{0}\subset T_{b^{r}}(x^{r})$ (as
$\varphi_{r}(t)=0$ if $t\in T_{0}$). This, together with \eqref{1.44}, leads
us to \eqref{3.17}.

Finally, appealing to \eqref{214}, \eqref{1.51} and (\ref{marco2}), we prove
\eqref{1.41} as follows:
\begin{align}
\lim_{r\rightarrow+\infty}\frac{\Vert b^{r}-\bar{b}\Vert_{\infty}^{q}}%
{d(x^{r},S(\bar{c},\bar{b}))}  & \leq\lim_{r\rightarrow+\infty}\frac{\Vert
v^{r}\Vert}{\bar{f}(x^{r})}\Vert b^{r}-\bar{b}\Vert_{\infty}^{q}\nonumber\\
& \leq\lim_{r\rightarrow+\infty}\frac{\Vert v^{r}\Vert}{\bar{f}(x^{r}%
)}(N+1)^{q}\bar{f}(x^{r})^{q}\label{marco10}\\
& =\lim_{r\rightarrow+\infty}(N+1)^{q}\left\{  \bar{f}(x^{r})^{q-1}\left\Vert
v^{r}\right\Vert \right\}  =0.\nonumber
\end{align}
The proof is complete.
\qed\end{proof}

\begin{remark}
The following example shows that, in the convex setting, the condition $%
x^{r}\in \mathcal{F}(b^{r})$
cannot be strengthened to $x^{r}\in \mathcal{%
S}_{\bar{c}}(b^{r})$ for the sequence $\{x^{r}\}_{r\in \mathbb{N}}$
in Proposition \ref{P3.2} as it happens in the linear case (see the proof
of \cite[Theorem 3.1]{CanHanParTol14}).

Consider the convex problem in $\mathbb{R}$: 
\begin{equation*}
\begin{aligned}
{\rm minimize}&\quad x^2 \\ {\rm subject\; to}&\quad x\leq
0.
\end{aligned}
\end{equation*}

Given $q\in \left( \frac{1}{2},1\right] $, take $\bar{c}=0,\ \bar{b}=0$, and
$\bar{x}=0$. Then $\mathcal{S}_{\bar{c}}(\bar{b})=\{0\}$ and the supremum
function $\bar{f}(x)=\sup \{x^{2},x\}$. Clearly, $\bar{f}(x)=x^{2}$ for $%
x\in (-\infty ,0]$. Then it is easy to verify that $\liminf\limits_{x%
\rightarrow \bar{x},\bar{f}(x)\downarrow 0}\bar{f}(x)^{q-1}d(0,\partial \bar{%
f}(x))=0$ and, thus, by Proposition \ref{P1}, the level set mapping $\mathcal{L}
$ is not $q$-order calm at $((0,0),0)\in \mathrm{gph}(\mathcal{L})$.
Moreover, there exist sequences $x^{r}:=-2^{-r}$ and $b^{r}:=2^{-2r}$ such
that
\begin{equation}
x^{r}=-2^{-r}\in \mathcal{F}(2^{-2r})=\mathcal{F}(b^{r})\quad \mbox{and}%
\quad \lim_{r\rightarrow +\infty }\frac{\Vert b^{r}-\bar{b}\Vert _{\infty
}^{q}}{d(x^{r},S_{\bar{c}}(\bar{b}))}=\lim_{r\rightarrow +\infty }\frac{1}{%
2^{(2q-1)r}}=0.
\end{equation}%
Recalling that $\mathrm{cone}(\emptyset )=\{0\}$, \eqref{3.17} also holds in
this setting. Obviously, $x^{r}\in \mathcal{F}(b^{r})$ but $x^{r}\notin
\mathcal{S}_{\bar{c}}(b^{r})=\{0\}$ for any $r\in \mathbb{N}$.

On the other hand, we have
\begin{equation}
\mathcal{S}_{\bar{c}}(b)=%
\begin{cases}
\{0\}\quad \forall b\in \lbrack 0,1), \\
\{b\}\quad \forall b\in (-1,0).%
\end{cases}
\label{5.27}
\end{equation}%
Noting that $\Vert b\Vert \leq \Vert b\Vert ^{\frac{2}{3}}$ for all $b\in
(-1,1)$, it readily follows from \eqref{5.27} that
\begin{equation}
d(x,\mathcal{S}_{\bar{c}}(\bar{b}))\leq \Vert b-\bar{b}\Vert ^{\frac{2}{3}%
}\quad \forall x\in \mathcal{S}_{\bar{c}}(b)\cap (-1,1)\;\mbox{and}\;b\in
(-1,1),
\end{equation}%
which guarantees that $\mathcal{S}_{\bar{c}}$ is $\frac{2}{3}$-order calm at
$(0,0)$.
\end{remark}

The above example reveals the fact that the $q$-order calmness of $\mathcal{S%
}$ at $((\bar{c},\bar{b}),\bar{x})$ may not imply the validity of the $q$%
-order calmness of $\mathcal{L}$ at $((f(\bar{x})+\langle \bar{c},\bar{x}%
\rangle ,\bar{b}),\bar{x})$. The following theorem constitutes a H\"{o}lder
convex counterpart of \cite[Theorem 3.1]{CanHanParTol14} for the linear case.

\begin{theorem}
Let $\bar{x}\in \mathcal{S}(\bar{c},\bar{b})$ and assume that $P(\bar{c},%
\bar{b})$ satisfies the Slater condition. Consider the following statements:%
\newline
(i) $\mathcal{S}$ is $q$-order calm at $((\bar{c},\bar{b}),\bar{x})$;\newline
(ii) $\mathcal{S}_{\bar{c}}$ is $q$-order calm at $(\bar{b},\bar{x})$;%
\newline
(iii) $\mathcal{L}$ is $q$-order calm at $((f(\bar{x})+\langle \bar{c},\bar{x%
}\rangle ,\bar{b}),\bar{x})$;\newline
(iv) $\bar{f}$ has a $q$-order local error bound at $\bar{x}$.\newline
Then $(iii)\Leftrightarrow (iv)\Rightarrow (i)\Rightarrow (ii)$ hold. In
addition, if $f$ and $g_{t}$ are linear, then $(i)\Leftrightarrow
(ii)\Leftrightarrow (iii)\Leftrightarrow (iv)$.
\end{theorem}

\begin{proof}
$(iii)\Leftrightarrow (iv)$ is Proposition \ref{P1}, while $%
(i)\Rightarrow (ii)$ is obvious. Now, we proceed by proving that $%
(iv)\Rightarrow (i)$. According to $(iv)$, there exist $\tau ,\delta \in
(0,+\infty )$ such that
\begin{equation*}
\tau d\left( x,\left[ \bar{f}\leq 0\right] \right) \leq \bar{f}%
_{+}(x)^{q}\quad \forall x\in B_\de(\bx).
\end{equation*}%
According to Lemma \ref{L1.2}, we may suppose that \eqref{1.25} holds for $%
U=B_\de(\bx)$, together with a certain neighborhood $V$ of $(%
\bar{c},\bar{b})$ and a certain $M>0$. Then, for all $(c,b)\in V$ and all $%
x\in \mathcal{S}(c,b)\cap U\cap \lbrack \bar{f}>0]$, it follow from (\ref%
{marco6}) that
\begin{align}
\tau d(x,\mathcal{S}(\bar{c},\bar{b})) =&\tau d(x,[\bar{f}\leq 0])\leq \bar{%
f}(x)^{q}  \notag \\
=&\left[ \sup \{f(x)-f(\bar{x})+\langle \bar{c},x-\bar{x}\rangle
;\;g_{t}(x)-\bar{b}_{t},\;t\in T\}\right] ^{q}  \label{1.55} \\
\leq &\left[ \sup \{f(x)-f(\bar{x})+\langle \bar{c},x-\bar{x}\rangle
;\;b_{t}-\bar{b}_{t},\;t\in T\}\right] ^{q},  \notag
\end{align}%
where we have used $x\in \mathcal{F}(b)$.

Let us take
\begin{equation*}
-(u+c)=\sum_{t\in T_{0}}\eta _{t}u_{t}
\end{equation*}%
for some finite subset $T_{0}\subset T_{b}(x)$, $u\in \partial f(x),\
u_{t}\in \partial g_{t}(x),$ and some $\eta _{t}>0,\ t\in T_{0}$, satisfying
$\sum_{t\in T_{0}}\eta _{t}\leq M$. Then we have
\begin{eqnarray}
-\langle u+c,x-\bar{x}\rangle  &=&\left\langle \sum_{t\in T_{0}}\eta
_{t}u_{t},x-\bar{x}\right\rangle =\sum_{t\in T_{0}}\eta _{t}\langle u_{t},x-%
\bar{x}\rangle   \notag \\
&\geq &\sum_{t\in T_{0}}\eta _{t}(g_{t}(x)-g_{t}(\bar{x}))\geq \sum_{t\in
T_{0}}\eta _{t}(b_{t}-\bar{b}_{t})  \label{marco9} \\
&\geq &-M\Vert b-\bar{b}\Vert _{\infty },  \notag
\end{eqnarray}%
which implies, from $u\in\partial f(x)$, that
\begin{align}
f(x)-f(\bar{x})+\langle \bar{c},x-\bar{x}\rangle \leq &\langle u+\bar{c},x-%
\bar{x}\rangle =\langle u+c,x-\bar{x}\rangle -\langle c-\bar{c},x-\bar{x}%
\rangle   \notag \\
\leq &M\Vert b-\bar{b}\Vert _{\infty }+\Vert c-\bar{c}\Vert \cdot \Vert x-%
\bar{x}\Vert .  \label{1.56}
\end{align}%
Recalling that $\Vert x-\bar{x}\Vert \leq \delta $ for all $x\in U$, %
\eqref{1.4} and \eqref{1.56} imply
\begin{equation*}
f(x)-f(\bar{x})+\langle \bar{c},x-\bar{x}\rangle \leq (M+\delta )\Vert
(c,b)-(\bar{c},\bar{b})\Vert ,
\end{equation*}%
and therefore \eqref{1.55} yields
\begin{equation}
\tau d(x,\mathcal{S}(\bar{c},\bar{b}))\leq \max \{(M+\delta )^{q},1\}\Vert
(c,b)-(\bar{c},\bar{b})\Vert ^{q},  \label{1.59}
\end{equation}%
whenever $x\in \mathcal{S}(c,b)\cap U\cap \lbrack \bar{f}>0]$. Observe that %
\eqref{1.59} is trivial for $x\in \lbrack \bar{f}\leq 0]=\mathcal{S}(\bar{c},%
\bar{b})$ and, hence, we have established $(i)$.

To finish the proof, we are establishing $(ii)\Rightarrow (iii)$ in the
linear setting. Suppose to the contrary that $\mathcal{L}$ is not $q$-order
calm at $((\bar{c},\bar{b}),\bar{x})$. To reach a contradiction, by
Proposition \ref{P3.2} it suffices to show that the sequence $x^{r}\in
\mathcal{F}(b^{r})$ in Proposition \ref{P3.2} is also contained in $\mathcal{%
S}_{\bar{c}}(b^{r})$, which readily follows from the KKT conditions (\ref%
{3.17}) in the linear setting (by continuity, it is not restrictive to
assume that $P(\bar{c},b^{r})$ satisfies the Slater condition).
\qed\end{proof}

Next we recall the so-called Extended N\"{u}rnberger Condition (ENC) \cite[%
Definition 2.1]{CanHanLopPar08}, which plays a crucial role in the present paper.

\begin{definition}
We say that ENC is satisfied at $((\bar{c},\bar{b}),\bar{x})\in \mathrm{gph}(%
\mathcal{S})$ when
\begin{eqnarray}
&&P(\bar{c},\bar{b})\;\text{satisfies the Slater condition and there is no}%
\;D\subset T_{\bar{b}}(\bar{x})  \notag \\
&&\text{with}\;|D|<n\;\text{such that}\;-(\partial f(\bar{x})+\bar{c}%
)\bigcap \left( \mathrm{cone}\Big(\bigcup_{t\in D}\partial g_{t}(\overline{x}%
)\Big)\right) \neq \emptyset .  \label{marco8}
\end{eqnarray}
\end{definition}

The following lemma is also crucial in our analysis; interested readers are
referred to \cite[Theorem~2.1 and Lemma 3.1]{CanHanLopPar08} for more details.

\begin{lemma}
\label{L6.1} Assume that ENC is satisfied at $((\bar{c},\bar{b}),\bar{x})\in
\mathrm{gph}(\mathcal{S})$. Then the following conditions hold:

$(i)$ $\mathcal{S}$ is single valued and Lipschitz continuous in a
neighborhood of $(\bar{c},\bar{b}).$

$(ii)$ If a sequence $\{((c^{r},b^{r}),x^{r})\}_{r\in \mathbb{N}}\subset
\mathrm{gph}(\mathcal{S})$ converges to $((\bar{c},\bar{b}),\bar{x})$, then $%
(b^{r},x^{r})\in \mathrm{gph}(\mathcal{S}_{\bar{c}})$ for $r$ large enough.
\end{lemma}

Thanks to Lemma \ref{L6.1}, we will arrive at the following theorem, which
shows that the parameter $c$ can be considered fixed in our analysis
provided that ENC is fulfilled at $((\bar{c},\bar{b}),\bar{x})\in \mathrm{gph%
}\mathcal{(S)}$.

\begin{theorem}
Let $((\bar{c},\bar{b}),\bar{x})\in \mathrm{gph}(\mathcal{S})$ and suppose
that ENC is satisfied at $((\bar{c},\bar{b}),\bar{x})$. Then
\begin{equation*}
\clmq\mathcal{S}((\bar{c},\bar{b}),\bar{x}) =\clmq\mathcal{S}_{\bar{c}}(\bar{b},\bar{x}).
\end{equation*}
\end{theorem}

\begin{proof}
According to Lemma \ref{L6.1}(i), we have
\begin{equation}
\clmq\mathcal{S}((\bar{c},\bar{b}),\bar{x}) =\lim_{r\rightarrow
+\infty }\frac{\Vert (c^{r},b^{r})-(\bar{c},\bar{b})\Vert ^{q}}{d(x^{r},%
\mathcal{S}(\bar{c},\bar{b}))},  \label{3.77}
\end{equation}%
for certain sequences $(c^{r},b^{r})\rightarrow(\bar{c},\bar{b})$ and $\{x^{r}\}=\mathcal{S}%
(c^{r},b^{r})$ with $x^{r}\neq \bar{x}$ and $x^{r}\rightarrow \bar{x}.$
By Lemma %
\ref{L6.1}(ii), we have
\begin{equation*}
\{x^{r}\}=\mathcal{S}_{\bar{c}}(b^{r})\quad \text{for }r\text{ large enough.}
\end{equation*}%
Therefore, \eqref{3.77} and the obvious consequence of (\ref{1.4})
\begin{equation*}
\Vert (c^{r},b^{r})-(\bar{c},\bar{b})\Vert \geq \Vert b^{r}-\bar{b}\Vert
_{\infty },
\end{equation*}%
ensure
\begin{equation*}
\clmq\mathcal{S}((\bar{c},\bar{b}),\bar{x})\geq
\liminf_{r\rightarrow +\infty }\frac{\Vert b^{r}-\bar{b}\Vert _{\infty }^{q}%
}{d(x^{r},\mathcal{S}(\bar{c},\bar{b}))}\geq \clmq\mathcal{S}_{\bar{c}}(\bar{b},\bar{x}).
\end{equation*}%
Since it readily follows from the definitions that $\clmq\mathcal{S}_{\bar{c}}(\bar{b},\bar{x})\geq \clmq\mathcal{S}((\bar{c},\bar{b}),\bar{x})$, the previous lower limit must be an ordinary limit and we
can obtain indeed
\begin{equation*}
\clmq\mathcal{S}((\bar{c},\bar{b}),\bar{x}) =\clmq\mathcal{S}_{\bar{c}}(\bar{b},\bar{x}) =\lim_{r\rightarrow +\infty }\frac{%
\Vert b^{r}-\bar{b}\Vert _{\infty }^{q}}{d(x^{r},\mathcal{S}(\bar{c},\bar{b}%
))}.
\end{equation*}%
The proof is complete.
\qed\end{proof}

In what follows, particularly in Theorem \ref{T6.02}, we consider a rather
weaker condition than ENC and then provide an upper estimate for $\mathrm{clm}_q[\mathcal{S},((\bar{c},\bar{b}),\bar{x})].$ To this aim, we associate
with $(b,x)\in \gph(\mathcal{S}_{\bar{c}})$ the family of KKT
subsets of $T$ given by
\begin{equation*}
\mathcal{K}_{b}(x):=\{D\subset T_{b}(x)\mid \left\vert D\right\vert \leq n\;%
\text{and}\;-(u+\bar{c})\in \mathrm{cone}\{\partial g_{t}(x),t\in D\}\;\text{%
for\ some}\;u\in \partial f(x)\}.
\end{equation*}%
For any $D\in \mathcal{K}_{\bar{b}}(\bar{x})$, we consider the supremum
function $f_{D}:\mathbb{R}^{n}\rightarrow \mathbb{R}$ given by
\begin{align}
f_{D}(x) :=&\sup \{g_{t}(x)-\bar{b}_{t},t\in T;\;-g_{t}(x)+\bar{b}_{t},t\in
D\}  \notag \\
=&\sup \{g_{t}(x)-\bar{b}_{t},\ t\in T\setminus D;\;|g_{t}(x)-\bar{b}%
_{t}|,\ t\in D\}.  \label{6.08}
\end{align}%
$\mathcal{K}_{b}(x)$ and\textbf{\ }$f_{D}(x)$ are convex counterparts of the
corresponding concepts in \cite[Section 3]{CanKruLopParThe14} for the linear model.

\begin{theorem}
\label{T6.02} Let $\mathcal{S}(\bar c, \bar b)=\{\bar x\}$ and assume that $%
P(\bar c, \bar b)$ satisfies the Slater condition. Then the following
estimate holds
\begin{equation*}
\liminf_{\substack{ b\rightarrow \bar{b}  \\ x\rightarrow \bar{x},\;x\in
\mathcal{F}(b)}}\frac{\|b-\bar b\|_\infty^q}{d(x, \mathcal{S}_{\bar c}(\bar
b))}\leq \inf_{D\in \mathcal{K}_{\bar b}(\bar x)}\liminf_{\substack{ x\to
\bar x  \\ f_D(x)>0}}f_D(x)^{q-1}d(0, \partial f_D(x)).
\end{equation*}
\end{theorem}

\begin{proof}The proof is based on similar arguments to those used in the
proof of \cite[Theorem 6]{CanKruLopParThe14}.
Picking a fixed $D\in \mathcal{K}%
_{\bar{b}}(\bar{x})$,
let us show that
\begin{equation}
\liminf_{\substack{ b\rightarrow \bar{b} \\ x\rightarrow \bar{x},\;x\in
\mathcal{F}(b)}}\frac{\Vert b-\bar{b}\Vert _{\infty }^{q}}{d(x,\mathcal{S}_{%
\bar{c}}(\bar{b}))}\leq \liminf_{\substack{ x\rightarrow \bar{x} \\ %
f_{D}(x)>0}}f_{D}(x)^{q-1}d(0,\partial f_{D}(x)).  \label{2.66}
\end{equation}%
We have
\begin{equation*}
\liminf_{\substack{ x\rightarrow \bar{x} \\ f_{D}(x)>0}}f_{D}(x)^{q-1}d(0,%
\partial f_{D}(x))=\lim_{r\rightarrow +\infty
}f_{D}(x^{r})^{q-1}d(0,\partial f_{D}(x^{r}))
\end{equation*}%
for a certain sequence $\{x^{r}\}_{r\in \mathbb{N}}$ such that $%
\lim_{r\rightarrow +\infty }x^{r}=\bar{x}$ and $f_{D}(x^{r})>0$ for all $%
r\in \mathbb{N}$. Obviously, $x^{r}\notin \mathcal{S}_{\bar{c}}(\bar{b})$
since $\mathcal{S}_{\bar{c}}(\bar{b})=\{\bar{x}\}$ and $f_{D}(\bar{x})=0$.
Note that $d(0,\partial f_{D}(x^{r}))>0$ since $x^{r}\notin$ {argmin}$_{x\in \mathbb{R}^{n}}f_{D}$.

To prove \eqref{2.66}, we need to build a new sequence of parameters $%
\{b^{r}\}_{r\in \mathbb{N}}\subset C(T,\mathbb{R}^{n})$ converging to $\bar{b%
}$ such that
\begin{equation}
x^{r}\in \mathcal{F}(b^{r})\quad \text{and}\quad \frac{\Vert b^{r}-\bar{b}%
\Vert _{\infty }^{q}}{\Vert x^{r}-\bar{x}\Vert }\leq \left( 1+\frac{1}{r}%
\right) ^{q} f_{D}(x^{r})^{q-1}d(0,\partial f_{D}(x^{r})).  \label{2.68}
\end{equation}%
First we give a lower bound for $\Vert x^{r}-\bar{x}\Vert$. If $u^{r}%
\in\partial f_{D}(x^{r})$, $u^{r}\neq0_{n},$ and
\[
\Vert x^{r}-\bar{x}\Vert\Vert u^{r}\Vert\geq\langle u^{r},x^{r}-\bar{x}%
\rangle\geq f_{D}(x^{r})-f_{D}(\bar{x})=f_{D}(x^{r}),
\]
and so
\begin{equation}
\Vert x^{r}-\bar{x}\Vert\geq\frac{f_{D}(x^{r})}{\Vert u^{r}\Vert}\geq
\frac{f_{D}(x^{r})}{d(0,\partial f_{D}(x^{r}))}.\label{2.70}%
\end{equation}
The next step consists of the construction of the desired sequence $\{b^{r}\}$
such that \eqref{2.68} holds. Once again we apply Urysohn's Lemma which guarantees
the existence of a certain function $\varphi_{r}\in C(T,[0,1])$ such that
\begin{equation}
\varphi_{r}(t)=\left\{
\begin{array}
[c]{ll}%
0, & \text{if}\;t\in D\\
1, & \text{if}\;g_{t}(x^{r})-\bar{b}_{t}\leq-(1+\frac{1}{r})f_{D}(x^{r}).
\end{array}
\right.  \label{5.12}%
\end{equation}
Recalling the definition of $f_{D}(x^{r})$ and the fact that $f_{D}(x^{r})>0$,
$D$ and $\{t\in T:\;g_{t}(x^{r})-\bar{b}_{t}\leq-(1+\frac{1}{r})f_{D}%
(x^{r})\}$ are disjoint closed sets in $T$. Certainly, if $t$ belongs to both
sets we reach the following contradiction:%
\[
f_{D}(x^{r})\geq\left\vert g_{t}(x^{r})-\bar{b}_{t}\right\vert \geq(1+\frac
{1}{r})f_{D}(x^{r}).
\]
If the set\textbf{ }$\{t\in T\mid g_{t}(x^{r})-\bar{b}_{t}\leq-(1+\frac{1}%
{r})f_{D}(x^{r})\}$ is empty we take\textbf{ }$\varphi_{r}(t)=0$ for all $t\in
T.$ Now, let us define, for each $t\in T$,
\[
b_{t}^{r}:=(1-\varphi_{r}(t))g_{t}(x^{r})+\varphi_{r}(t)(\bar{b}_{t}%
+f_{D}(x^{r})).
\]
For each $r$, the definition of $b^{r}$ and \eqref{6.08} clearly imply that
\[
b_{t}^{r}-g_{t}(x^{r})=\varphi_{r}(t)(f_{D}(x^{r})+\bar{b}_{t}-g_{t}%
(x^{r}))\geq0
\]
and thus $x^{r}\in\mathcal{F}(b^{r})$. Finally, let us observe that $b_{t}%
^{r}-\bar{b}_{t}=f_{D}(x^{r})$ when $\varphi_{r}(t)=1,$ and
\[
-(1+\frac{1}{r})f_{D}(x^{r})<g_{t}(x^{r})-\bar{b}_{t}\leq f_{D}(x^{r})
\]
when $\varphi_{r}(t)<1$. Accordingly,
\[
\Vert b^{r}-\bar{b}\Vert_{\infty}\leq(1+\frac{1}{r})f_{D}(x^{r})
\]
which, together with \eqref{2.70}, entails
\[
\frac{\Vert b^{r}-\bar{b}\Vert_{\infty}^{q}}{\Vert x^{r}-\bar{x}\Vert}%
\leq\frac{d(0,\partial f_{D}(x^{r}))}{f_{D}(x^{r})}\Vert b^{r}-\bar{b}%
\Vert_{\infty}^{q}\leq\left(  1+\frac{1}{r}\right)  ^{q}f_{D}(x^{r}%
)^{q-1}d(0,\partial f_{D}(x^{r})),
\]
which ensures \eqref{2.68}.
\qed\end{proof}

Finally, we will consider the linear counterpart of $P(c,b)$; namely, we
will always assume that $f=0$ and $g_{t}(x)=\langle a_{t},x\rangle $ for all
$t\in T$ therein, where $t\mapsto a_{t}\in \mathbb{R}^{n}$ is continuous on $%
T$.

\begin{corollary}
Let $\mathcal{S}(\bar c,\bar b)=\{\bar x\}$ and assume that $P(\bar c, \bar
b)$ satisfies the Slater condition. Then the following estimates hold
\begin{equation}
\clmq\mathcal{S}((\bar c,\bar b),\bar x)\leq \clmq\mathcal{S}_{\bar c}(\bar b,\bar x)\leq \inf_{D\in \mathcal{K}_{\bar
b}(\bar x)}\liminf_{\substack{ x\to \bar x  \\ f_D(x)>0}}f_D(x)^{q-1}d(0,
\partial f_D(x)).  \label{6.21}
\end{equation}
\end{corollary}

\begin{proof}The first inequality follows straightforwardly from %
\eqref{1.6}. To prove \eqref{6.21}, by Theorem \ref{T6.02} it suffices to
show
that
the sequence $x^{r}\in \mathcal{F}(b^{r})$ produced in Theorem \ref%
{T6.02} is also contained in $\mathcal{S}_{\bar{c}}(b^{r})$. Taking a fixed $%
D$ as in Theorem \ref{T6.02}, since for all $t\in D$, by \eqref{5.12} we
could have $\varphi _{r}(t)=0$, which follows from the definition of $%
b_{t}^{r}$ that $b_{t}^{r}=g_{t}(x^{r})$ and then implies that $D\subset
T_{b^{r}}(x^{r})$. Noting that $f$ and $g_{t}$ are linear functions, we
obtain $D\in \mathcal{K}_{b^{r}}(x^{r})$. Recalling that $x^{r}\in \mathcal{F%
}(b^{r})$, this certainly yields that $x^{r}\in \mathcal{S}_{\bar{c}}(b^{r})$%
.\qed\end{proof}

\section*{Acknowledgement}

The authors would like to thank Prof.
{Xi Yin}
Zheng for his valuable comments which
helped us
improve the presentation of the paper
{and Dr Li Minghua who has read the preliminary version of the manuscript and found a mistake in one of the proofs}.
We are grateful to PhD students Bui Thi Hoa and Nguyen Duy Cuong from Federation University Australia for carefully reading the manuscript, eliminating numerous glitches and
{contributing to Example~\ref{E3.20}.}

\end{document}